\documentclass[12pt,reqno]{amsart}

\usepackage{amsmath}
\usepackage{amssymb}
\usepackage{amsmath,amssymb}
\usepackage{pgfplots}
\pgfplotsset{compat=1.18}
\usepackage{amstext}
\usepackage[T1]{fontenc}
\usepackage[utf8]{inputenc}
\usepackage{float}
\usepackage{microtype}
\usepackage[a4paper%,top=27mm,bottom=28mm,left=29mm,right=29mm
]{geometry}
\usepackage{amsmath,amssymb,mathtools}
\usepackage{enumitem}
\usepackage[numbers,sort&compress]{natbib}
\usepackage[dvipsnames]{xcolor}
\usepackage{hyperref}
\usepackage{bookmark}
\usepackage{a4wide}
\usepackage{graphicx}
\usepackage{tikz}
\allowdisplaybreaks \numberwithin{equation}{section}
\usepackage{color}
\usepackage{cases}

\hypersetup{
  colorlinks=true,
  linkcolor=MidnightBlue,
  citecolor=MidnightBlue,
  urlcolor=MidnightBlue,
  pdfauthor={},
  pdftitle={Asymmetric stationary  Water Waves in Finite Depth: Expanded Introduction}
}

\setlist[enumerate]{leftmargin=2.15em,itemsep=.45em,topsep=.45em}
\setlist[itemize]{leftmargin=1.85em,itemsep=.25em,topsep=.35em}
\allowdisplaybreaks

\newcommand{\R}{\mathbb{R}}

\numberwithin{equation}{section}

\newtheorem{theorem}{Theorem}[section]
\newtheorem{proposition}[theorem]{Proposition}

\newtheorem{lemma}[theorem]{Lemma}

\theoremstyle{definition}

\theoremstyle{remark}
\newtheorem{remark}[theorem]{Remark}

\newcommand{\ep}{\varepsilon}

\begin{document}

\title
[Asymmetric Water Waves]{Asymmetric Periodic Water Waves in Finite Depth}

\author{Yuchen Wang}
\address{School of Mathematics, Tianjin Normal Univeristy \\ Tianjin, China \&
 Institute of Mathematics, University of Augsburg, Augsburg, Germany}
\email{wangyuchen@mail.nankai.edu.cn}
	
\author{Shusen Yan}
\address{ School of Mathematics and Statistics, Key Laboratory of Nonlinear Analysis and Applications
		(Ministry of Education), Central China Normal University, Wuhan 430079, China }
\email{syan@ccnu.edu.cn}

\author{Maolin Zhou}
\address{Chern Institute of Mathematics and LPMC, Nankai University, Tianjin, China}
\email{zhouml123@nankai.edu.cn}

%\thanks{This work is partially supported by ARC}

\begin{abstract}
The classical existence theory for periodic traveling water waves has been developed largely within the class of reflection-symmetric solutions. For the two-dimensional water-wave equations in finite depth, we construct periodic traveling waves whose free surfaces admit no axis of reflection symmetry. Our results cover both pure-gravity and capillary-gravity cases, solving a long-standing question concerning the existence of asymmetric periodic traveling water waves. In the pure-gravity case, the novelty is the introduction of a background flow with constant vorticity to recover the solvability of the linearized operator for the free surface of the water wave. The asymmetric water waves we construct contain vortices concentrating at asymmetric critical points of the corresponding Kirchhoff-Routh type functions, giving rise to the asymmetry of the free surfaces.
\end{abstract}
\keywords{Water waves, asymmetric traveling waves,
free boundary problems}

\subjclass[2020]{Primary 76B15;
Secondary 35R35, 76B47, 76B45}
\maketitle
\section{Introduction}

Despite more than two centuries of research on water waves, the
existence theory for periodic traveling waves has largely focused on
symmetric solutions. Many bifurcation constructions are carried out
in spaces of even functions, imposing additional reflection symmetry. The reflection invariance of the Euler equations does not imply that every solution is symmetric. A basic question is
therefore whether these equations admit periodic traveling  waves
whose free surfaces have no axis of reflection.

Following the numerical work of Zufiria
\cite{Zufiria1987Deep,Zufiria1987Capillary}, Okamoto
\cite[p.~91]{Okamoto1990} posed the problem of proving the existence
of such waves. Constantin later included this question among the
open problems in water-wave theory \cite[p.~128]{Constantin2011}.
Further numerical investigations have produced asymmetric periodic
profiles \cite{GaoWangVandenBroeck2017}. Related observations of
wave-profile asymmetry have been reported in laboratory wave tanks
\cite{PerlinTing1992,PerlinLinTing1993,Zhang1995,Caulliez2013}
and in natural environments, including coastal waters and rivers
\cite{ElgarFreilichGuza1990,Zhang1999Field,MartinsEtAl2017}.
These observations arise in different physical settings; the
existence of asymmetric periodic traveling waves for the full Euler
free-boundary problem requires a separate mathematical argument.

In this paper, we construct asymmetric periodic traveling waves
with vorticity for the full two-dimensional Euler free-boundary
equations in finite depth. The construction applies to both
capillary-gravity waves and pure-gravity waves. To the best of our
knowledge, this is the first rigorous existence result for such
waves, giving an affirmative answer to the existence question in
the class of rotational flows. We construct waves with concentrated
vorticity near an asymmetric three-vortex configuration and prove
that their free surfaces have no axis of reflection.

\subsection*{Formulation of the water-wave problem}

For
$\eta\in C^{2,\beta}(\mathbb{R})$, with $0<\beta<1$, let
\[
\Omega=\Omega_{\eta}
:=\left\{(y_1,y_2)\in\mathbb{R}^2:
0<y_2<\pi+\eta(y_1)\right\}.
\]
The motion of the fluid is governed by the steady incompressible
Euler equations
\begin{equation}\label{080-24-10}
	\begin{aligned}
		(\mathbf{u}\cdot\nabla)\mathbf{u}+\nabla P
		+g\boldsymbol{e}_2&=0
		&&\text{in }\Omega,\\
		\nabla\!\cdot\mathbf{u}&=0
		&&\text{in }\Omega,
	\end{aligned}
\end{equation}
where $\mathbf{u}=(u_1,u_2)$ is the fluid velocity, $P$ is the pressure, $g>0$ is the gravitational
acceleration, and $\boldsymbol{e}_2=(0,1)$ is the upward unit vector.
The fluid density is normalized to one.

The free surface is
\[
S:=\left\{\bigl(y_1,\pi+\eta(y_1)\bigr):
y_1\in\mathbb{R}\right\},
\]
and the flat impermeable bottom is given by $y_2=0$. Throughout this
paper, we require
\[
\pi+\eta(y_1)\geq\frac{\pi}{2}
\qquad\text{for all }y_1\in\mathbb{R},
\]
so that the free surface remains uniformly separated from the bottom.
The kinematic boundary conditions are
\begin{equation}\label{081-24-10}
	\mathbf{u}\cdot\boldsymbol{n}=0
	\quad\text{on }S,
	\qquad
	u_2=0
	\quad\text{on }y_2=0,
\end{equation}
where
\[
\boldsymbol{n}
=\frac{(-\eta_{y_1},1)}{\sqrt{1+\eta_{y_1}^2}}
\]
is the outward unit normal to the free surface. The dynamic boundary
condition is
\begin{equation}\label{082-24-10}
	P=\alpha^2\kappa+P_{\mathrm{atm}}
	\quad\text{on }S,
\end{equation}
where
\[
\kappa=-\frac{\eta_{y_1y_1}}
{(1+\eta_{y_1}^2)^{3/2}}
\]
is the signed curvature, $\alpha^2\geq0$ is the coefficient of
surface tension, and $P_{\mathrm{atm}}$ is the constant atmospheric
pressure. We measure pressure relative to $P_{\mathrm{atm}}$ and
hence normalize $P_{\mathrm{atm}}=0$.

Equations \eqref{080-24-10}-\eqref{082-24-10} describe the
\emph{stationary pure-gravity water-wave problem} when $\alpha=0$
and the \emph{stationary capillary-gravity water-wave problem}
when $\alpha>0$. Note that traveling waves can equivalently be viewed as stationary waves in coordinates moving at a constant speed. The local well-posedness of the corresponding
time-dependent equations with vorticity has been studied
extensively, both with surface tension and, under the Taylor-sign
condition, without surface tension; see, for example,
\cite{ChristodoulouLindblad2000,Lindblad2005,CoutandShkoller2007,
	ShatahZeng2008Geometry,ShatahZeng2011}.

A solution $(\mathbf{u},P,\eta)$ is called a \emph{periodic
	stationary water wave} if there exists $L>0$ such that
\[
\begin{aligned}
	\eta(y_1+L)&=\eta(y_1),\\
	\mathbf{u}(y_1+L,y_2)&=\mathbf{u}(y_1,y_2),\\
	P(y_1+L,y_2)&=P(y_1,y_2)
\end{aligned}
\]
throughout their respective domains. The flat bottom makes the problem
invariant under horizontal translations. A stationary wave is
\emph{symmetric} if, after a suitable horizontal translation,
$\eta$, $P$, and $u_1$ are even in $y_1$, while $u_2$ is odd.
Otherwise, it is called \emph{asymmetric}. 

\bigskip

\subsection*{Background and related work}

The classical theory of periodic traveling water waves goes back
to the work of Stokes, Nekrasov, Levi-Civita, and Struik
\cite{Stokes1847,Nekrasov1921,LeviCivita1925,Struik1926}
on irrotational waves. Subsequent developments in the existence
theory and the study of extreme waves include
\cite{Krasovskii1962,KeadyNorbury1978,AmickFraenkelToland1982,
	Plotnikov1982,VarvarucaWeiss2011}; see also the references therein.

For rotational flows, Constantin and Strauss
\cite{ConstantinStrauss2004} constructed global connected sets of
symmetric periodic gravity waves with general vorticity and no
stagnation points. In the constant-vorticity setting, Constantin
and V\u{a}rv\u{a}ruc\u{a} \cite{ConstantinVarvaruca2011}
established regularity and local bifurcation results allowing
interior stagnation points. Constantin, Strauss, and
V\u{a}rv\u{a}ruc\u{a} \cite{ConstantinStraussVarvaruca2016}
developed a global bifurcation theory that also allows critical
layers. Multimodal waves and waves with multiple critical layers
were constructed in
\cite{EhrnstromEscherWahlen2011,EhrnstromWahlen2015,
	KozlovLokharu2018}. These constructions take place in
reflection-symmetric spaces. Complementary symmetry theorems show
that symmetry follows from suitable assumptions on the flow or
the wave profile; see, for example,
\cite{Garabedian1965,ConstantinEhrnstromWahlen2007,
	MatiocMatioc2013,Tulzer2012}.

Rigorous constructions of asymmetric traveling waves have also been obtained in related settings. M{\ae}hlen and Svensson Seth~\cite{MaehlenSvenssonSeth2024} employed a Lyapunov-Schmidt reduction to construct asymmetric periodic traveling waves of arbitrarily small amplitude for the capillary-gravity Whitham equation, a simplified model of water waves. More recently, Svensson Seth~\cite{Seth2025} developed a general bifurcation framework that identifies necessary and sufficient conditions for symmetry breaking through the interaction of two resonant Fourier modes. Applying this approach to the infinite-depth capillary-gravity Whitham and Babenko equations, he obtained conditional nonexistence results for small-amplitude asymmetric waves and pointed out the possibility of asymmetric bifurcation in finite depth. In a forthcoming work, Buffoni, Dahne, Ehrnstr\"om, and Svensson Seth~\cite{Buffoni2026} are investigating the existence of asymmetric periodic solutions to the capillary-gravity water-wave equations through a detailed analysis of high-order resonant interactions and matching the nonlinear terms in a very high order, supplemented by computer-assisted estimates.  

Moreover, \cite{KozlovLokharu2018,KozlovLokharu2019} studied small-amplitude
steady gravity waves with vorticity in finite depth. Their
finite-dimensional Hamiltonian reduction yields asymmetric steady
waves, but does not establish their spatial periodicity.

Thus,
neither result provides the periodic asymmetric solutions of the
two-dimensional Euler free-boundary problem considered here.

Ehrnstr\"om, Walsh, and Zeng \cite{EhrnstromWalshZeng2023}
introduced a different approach by relating the stationary
water-wave problem to singularly perturbed elliptic equations.
They constructed symmetric solitary capillary-gravity waves in
finite depth with exponentially localized vorticity. Their method
provides the starting point for our construction, while the
absence of symmetry and the inclusion of pure-gravity waves
require additional analysis.

\bigskip

\subsection*{Main result}

Our main theorem establishes the existence of asymmetric periodic
stationary waves with vorticity for every nonnegative value of the
surface-tension parameter. For a discussion on the existence of asymmetric solitary waves, we refer to Appendix B.

\begin{theorem}\label{th1-1}
For every $\alpha\geq0$, there exists a periodic stationary
solution $(\mathbf{u},P,\eta)$ of \eqref{080-24-10},
\eqref{081-24-10}, and \eqref{082-24-10}, such that for any $y^* \in \R$
up to any horizontal translation, none of equations 
\[
P(y^*-y_1, y_2)=P(y^*+y_1, y_2), \quad \eta(y^*-y_1)=\eta(y^*+y_1), 
\]
and 
\[
\begin{split}
u_1(y^*-y_1,y_2) & = u_1(y^*+y_1,y_2), \\ u_2(y^*-y_1,y_2) & = - u_2(y^*+y_1,y_2), \quad
\end{split}
\]
hold for all $(y_1,y_2) \in \Omega_\eta$.
\end{theorem}

To prove Theorem~\ref{th1-1}, we will construct solutions, such that
the stream functions concentrate at some asymmetric  critical points.
When using the techniques in singularly perturbed elliptic equations to
construct asymmetric periodic waves, some essential difficulties arise.
In \cite{EhrnstromWalshZeng2023},
solitary capillary-gravity wave with one vortex was constructed.
The problem is invariant under the horizontal translation.
By working in a space consisting of even functions in $y_1$, the
location of the vortex is automatically put in the line $y_1=0$ and
only one parameter in the vertical direction needs to be determined.
However, without the symmetry assumption, the translation invariance
leads to the degeneracy of the problem, which makes the implicit
functional theorem difficult to apply.

The vorticity function chosen in \cite{EhrnstromWalshZeng2023} has the
form $\ep^{-2}( u^p-u)$ for some $p>1$, where $u$ is the stream
function, and $\ep>0$ is a small parameter.
Thus, the vorticity is not zero almost everywhere, though it is
exponentially small in $\ep^{-1}$ when $y$ is away from the center
of the vortex.
For asymmetric waves, it is natural to look for
waves with several vortices, locating at points
$x_{1},\cdots, x_k$ without horizontal symmetry.
If the vorticity function is chosen as $ \ep^{-2}( u^p-u)$, the existence of such solutions depends
on the existence of "stable critical" point of the following function

\begin{equation}\label{1-9-9}
\mathcal P(x_{1},\cdots, x_k)
=\sum_{j=1}^k e^{-\min(x_{j2}, \pi-x_{j2})\ep^{-1}}
-B\sum_{i\ne j}(-1)^{\tau_i+\tau_j}e^{-|x_j-x_i|\ep^{-1}},
\end{equation}
where $x_j=(x_{j1}, x_{j2})$, $B>0$ is a constant, and
$\tau_j$ is $0$ or $1$, depending on the rotational direction of
the vortex at $x_j$.
The study of the existence of "stable critical" point for
$\mathcal P$ is not trivial because firstly, $\mathcal P$ depends
on $\ep$, and secondly the function $\min(x_{j2}, \pi-x_{j2})$
is not $C^1$.
The case $k=1$ is a bit simple.
The "critical" point of $\min(t, \pi-t)$ is $t=\frac\pi2$,
which is exactly the non-differentiable point of this function.

In this paper,
 the vorticity function has the following form
\[
\sum\limits_{j=1}^k 1_{B_\delta(x_{0,j})} f_j(u),
\]
where $1_S$ is the characteristic function of the set $S$,
$u$ is the stream function,
$f_j(u)=\ep^{-2} (u-\kappa_j)_+^p$, or
$f_j(u)= -\ep^{-2}(-u-\kappa_j)_+^p$,
depending on the rotational direction,
and $\kappa_j>0$ is a given constant.
For this choice, the position
of the points $x_{1},\cdots, x_k$ must be the critical points of
the Kirchhoff-Routh functions $\tilde{\mathcal K_p}$,
defined in \eqref{7-26-8}.
Unlike the function defined in \eqref{1-9-9},
$\tilde{\mathcal K_p}$ is independent of $\ep$ and smooth.
More importantly, the different choice of $\kappa_j$, or
equivalently, the different choice of the local vorticity strength
at $x_j$, will affect the existence of non-symmetric critical
points of $\tilde{\mathcal K_p}$.
It is exactly this flexibility that makes the construction of
asymmetric waves possible.

There is a striking difference between the
capillary-gravity case and the pure-gravity case.
In the case $\alpha>0$, the main term in the
Bernoulli equation is given by
\[
    g-\alpha^2\partial_{y_1}^2,
\]
which is invertible and gives good control of the derivatives of $\eta$.
However, in the pure-gravity case $\alpha=0$,
the linearized operator for the
Bernoulli equation gives no control of the derivatives on $\eta$, and thus 
makes the use of the implicit function theorem unworkable.
To change this bad situation,
our novelty is to introduce a background
flow with constant-vorticity $\gamma$, so that the free boundary
equation is changed to the problem of finding a harmonic function,
satisfying a nonlocal mixed boundary condition.
We will prove that except for a sequence of $\gamma$, this problem
is uniquely solvable and the Schauder estimate holds.
Nevertheless,The approximations of the stream function and the free boundary must be chosen together in order to obtain better
approximation.
This strong coupling is the main analytical challenge  in the
pure-gravity construction.

The vortex centers are governed by a Kirchhoff-Routh type function.
For solitary waves, which can be regarded as the case that the period
is infinity, in view of the non-existence result in
\cite{GrossiTakahashi2010},
there is no solitary wave, consisting of several vortices, all
of which rotate in the same direction.
On the other hand, for the periodic wave, such profile does exist.
This shows that the periodicity enlarges the possible vortex
configurations.

Though we believe that more vortices should give rise to many
more interesting types of configurations, in this paper, we focus on
a simple case.
We choose a configuration of three vortices locating
at some non-symmetric points. This is sufficient
for us to find an asymmetric periodic wave.

The stability of the waves constructed here remains to be
investigated. Their concentrated vorticity and its interaction
with the free surface introduce difficulties beyond those of
the irrotational problem. For small-amplitude Stokes waves,
Berti, Maspero, and Ventura
\cite{BertiMasperoVentura2022,BertiMasperoVentura2023}
established linear modulational instability and gave precise
descriptions of the unstable eigenvalues near the origin in
the infinite-depth and finite-depth settings, in the relevant
unstable parameter regimes; see also
\cite{NguyenStrauss2023}. Chen and Su \cite{ChenSu2023}
proved nonlinear modulational instability using a nonlinear
Schr\"odinger approximation. These results provide useful
context for the stability problem for asymmetric rotational
waves.

The remainder of the paper is organized as follows.
In Section~2, we explain the main idea in the construction of
the solutions. 
Section~3 establishes the required estimates for these approximate
solutions. In Section~4, we study the invertibility of
the corresponding linearized operator, while in   Section~5,
we  carry  out the nonlinear
Lyapunov-Schmidt reduction, and Section~6 completes the proof
of the main theorem by solving the reduced equations for the
vortex centers. Appendix~A establishes solvability and Schauder
estimates for the nonlocal mixed boundary problem.
Appendix~B studies the critical points and
nondegeneracy of the limiting and periodic Kirchhoff-Routh functions.
In
Appendix~C,  the leading term of free-surface  in the
pure-gravity case is presented.

\section{Outline of the construction}

{By the Hodge decomposition theorem, a planar divergence-free velocity field in a simply-connected domain can be globally represented as $\mathbf{u} = \nabla^\perp \psi$. Such $\psi$,
which is unique up to a constant,  is referred to as the stream function and the scalar vorticity is subsequently represented by $\omega = - \Delta \psi$. Then the stationary free-boundary Euler equations in the fixed coordinate frame can be written as:
\begin{equation}\label{eq1-6}
\nabla^\bot\psi \cdot \nabla \Delta \psi=0\,\,\,\text{in}\,\,\Omega, \qquad \psi = 0 \,\,\, \text{ on } \partial \Omega.
\end{equation}
Moreover, since the total dynamical pressure is conserved along the streamline in steady Euler flows, in particular the water-air interface, we have the Bernoulli equation 
\begin{equation}\label{10-18-11}
	\frac12 \lvert \nabla \psi \rvert^2 + g \left(\pi + \eta \right) + \alpha^2 \kappa = C \quad \text{ on } \quad x_2 = \pi+ \eta(x_1),
\end{equation}
which encompasses the geometry of the water-air interface, where $C$ is a constant that will be specified later.}
%The constant $g\approx 9.81$ is the gravitational acceleration,
%$\alpha\ge 0$ is a constant measuring the surface tension, and
%\[
%\kappa=-\frac{\eta''}{(1+ (\eta')^2)^{\frac32}}
%\]
%is the signed curvature of the upper boundary of $\Omega$.

In this paper, we aim to prove the existence of asymmetric solutions
$\psi$, $\omega$ and $\eta$ of
\eqref{eq1-6}-\eqref{10-18-11}, which are periodic in $y_1$.

Let $L>0$ be a given constant. Denote
\[
D_L=\bigl\{ y=(y_1, y_2):\; 0<y_1<L,\;  0<y_2<\pi\bigr\}.
\]
Theorem \ref{th1-1} has an equivalent version as follows:

\begin{theorem}\label{th2-1}
For any $\alpha\geq0$ and any $L>0$, there exist
functions $\psi$, $\omega$ and $\eta$, which are $L$-periodic
and asymmetric  in $y_1$, such that they satisfy
\eqref{eq1-6}-\eqref{10-18-11}
for a suitable constant $C$.
Moreover, $\psi\in C^{2,\beta}(D_L)$ and $\eta\in C^{2,\beta}([0,L])$
for any $\beta\in (0, 1)$.
\end{theorem}

Let $\Gamma= \Gamma_1+ i\Gamma_2$ be a holomorphic function from $D$ to
$\mathbb C$ and $\Gamma_2(y_1, 0)=0$.
We assume that $\Omega$ is represented by
\[
\Omega=(\text{id} +\Gamma) D= \{(y_1+\Gamma_1(y), y_2+\Gamma_2(y)): \; y_1\in\mathbb R,\  0<y_2<\pi\},
\]
and let $u(z)= \psi(z +\Gamma(z))$. Then
the equation $-\Delta \psi =\omega$ becomes
\begin{equation}\label{eq1-2}
\begin{cases}
-\Delta u= |1+\Gamma'|^2 \omega( (y_1+\Gamma_1(y), y_2+\Gamma_2(y))),& y\in D,\\
u(y_1, 0)=u(y_1, \pi)=0,
\end{cases}
\end{equation}
and \eqref{10-18-11} can be written as
\begin{equation}\label{eq1-3}
\frac12 \frac{1}{|1+\Gamma'|^2} \bigl(\frac{\partial u}{\partial y_2}\bigr)^2
-\alpha^2 \frac{Im (\Gamma'' (1+\bar {\Gamma'}))  }{ |1+\Gamma'|^3  } +g(\pi +\Gamma_2)=C,\quad \text{on}\; \{y_2=\pi\}.
\end{equation}

We will construct $u$, $\omega$ and $\Gamma$ satisfying
\eqref{eq1-2}-\eqref{eq1-3}, such that all these functions are periodic
in $y_1$ with period $L>0$. Here, the vorticity $\omega$ depends on the
stream function $\psi$, or equivalently on $u$.
On the other hand, noting that $\Gamma_1$ is a harmonic conjugate of
$\Gamma_2$, we see that, up to a constant, $\Gamma_1$ is uniquely
determined by $\Gamma_2$. On the other hand, if $\Gamma_2$
is periodic in $y_1$, the derivatives
$\frac{\partial \Gamma_1}{\partial y_j}$, $j=1, 2$,
are also periodic in $y_1$.
However, to ensure that $\Gamma_1$ has the same periodicity in $y_1$,
we need to impose that  $\Gamma_2$ satisfies
\begin{equation}\label{neq1-3}
\int_{0}^L \Gamma_2(t, \pi)\,dt=0.
\end{equation}
This can be achieved if we choose $C$ as follows:
\begin{equation}\label{1-23-8}
C=\frac1L \int_{0}^L \Bigl( \frac12 \frac{1}{|1+\Gamma'|^2} \bigl(\frac{\partial u}{\partial y_2}\bigr)^2
-\alpha^2 \frac{Im (\Gamma'' (1+\bar {\Gamma'}))  }{ |1+\Gamma'|^3  }   \Bigr)\,dy_1-
\frac aL\int_{0}^L \Gamma_2(t, \pi)\,dt+ g\pi,\quad y_2=\pi,
\end{equation}
with $a\ne -g$.

In the following, we will formulate the partial differential equation
for \eqref{eq1-2} and discuss the linearized problem for \eqref{eq1-3}.
It turns out that
in the pure-gravity case $\alpha=0$, it is convenient to introduce a background
flow  to recover {the solvability of} the linearized operator \eqref{eq1-3}.
By contrast, in the capillary-gravity case $\alpha>0$, the
surface-tension term  provides the desired linear second order derivative,
so no background flow is needed.

\bigskip
\subsection{The case $\alpha > 0$.}

We first discuss \eqref{eq1-2} and look for a solution $u$ of \eqref{eq1-2} such that the
corresponding vorticity $\omega$ has support near some given points
$x_{0, j}\in D_L$, $j=1,\cdots, k$.
For this purpose, we take
\begin{equation}\label{1-24-8}
\omega( (y_1+\Gamma_1(y), y_2+\Gamma_2(y))) =\frac1{\ep^2}\sum\limits_{j=1}^k 1_{B_\delta(x_{0,j})} f_j(u),
\end{equation}
where $\ep>0$ is a small parameter,
and $1_S=1$ in $S$, while $1_S=0$ in $\mathbb R^2\setminus S$.
For simplicity, we take
\[
f_j(u)= (u-\kappa_j)_+^p,
\]
or
\[
f_j(u)= -(-u-\kappa_j)_+^p,
\]
where $p>1$, and $\kappa_j>0$ is a given constant.

We aim to find a solution $(u_\ep, \Gamma_\ep)$ of
\eqref{eq1-2}-\eqref{eq1-3} and \eqref{1-24-8} such that
\begin{itemize}
\item $(u_\ep, \Gamma_\ep)$ is periodic in $y_1$ with period $L$;

\item $f_j(u_\ep)=0$ in $D_L\setminus \cup_{j=1}^k B_\delta(x_{0,j})$,
and $\bigl(\frac{\partial u_\ep}{\partial y_2}\bigr)^2$ is small
on $\{ y_2=\pi\}$;

\item $
|\Gamma|+|\Gamma'|<<1.
$
\end{itemize}
Thus, we will solve the following problem
\begin{equation}\label{10-24-8}
\begin{cases}
-\Delta u= |1+\Gamma'|^2\frac1{\ep^2}\sum\limits_{j=1}^k 1_{B_\delta(x_{0,j})} f_j(u) ,\;
\; y\in D_L,\\
u(y_1, 0)=u(y_1, \pi)=0,\;\; \text{$u$ is periodic in $y_1$ in $D_L$}.
\end{cases}
\end{equation}

Since our aim is to look for $\Gamma$ with $|\Gamma'|<<1$,
\eqref{10-24-8} can be viewd as a perturbation of the following problem
\begin{equation}\label{24-24-8}
\begin{cases}
-\Delta u=\frac1{\ep^2}\sum\limits_{j=1}^k 1_{B_\delta(x_{0,j})} f_j(u) ,\;
\; y\in D_L,\\
u(y_1, 0)=u(y_1, \pi)=0,\;\; \text{$u$ is periodic in $y_1$ in $D_L$},
\end{cases}
\end{equation}

\bigskip

Next, let us discuss the free boundary equation \eqref{eq1-3}.

Since our goal is to find $u$ and $\Gamma$ such that 
$
|\Gamma|+|\Gamma'|<<1, 
$
we expand \eqref{eq1-3} to obtain
\begin{equation}\label{60-24-8}
\frac12  \bigl(\frac{\partial u}{\partial y_2}\bigr)^2-\alpha^2 \eta'' -
\mathcal F_1(u, \Gamma)+g(\pi+\eta)=C,\quad \text{on}\; \{y_2=\pi\},
\end{equation}
where $\eta (t)  = \Gamma_2(t,\pi) $, and
$\mathcal F_1(u, \Gamma)$ satisfies
\[
|\mathcal F_1(u, \Gamma)|=O\Bigl(  \bigl(\frac{\partial u}{\partial y_2}\bigr)^2|\Gamma'|+
|\Gamma''||\Gamma'|\Bigr).
\]

On the other hand, the constant $C$ given in \eqref{1-23-8} satisfies
\[
C= g\pi +\frac1{2L} \int_0^L \bigl(\frac{\partial u}{\partial y_2}\bigr)^2-\frac1{L} \int_0^L
\mathcal F_1(u, \Gamma)
\]
if we choose $a=0$.
Thus \eqref{eq1-3} becomes
\begin{equation}\label{61-21-8}
-\alpha^2 \eta'' +g\eta=-\frac12  \bigl(\frac{\partial u}{\partial y_2}\bigr)^2+
\frac1{2L} \int_0^L \bigl(\frac{\partial u}{\partial y_2}\bigr)^2+
\mathcal F_1(u, \Gamma)-\frac1{L} \int_0^L
\mathcal F_1(u, \Gamma),\quad \text{on}\; \{y_2=\pi\},
\end{equation}
which is a perturbation of the following equation:
\begin{equation}\label{10-21-8}
-\alpha^2 \eta'' +g\eta=-\frac12  \bigl(\frac{\partial u}{\partial y_2}\bigr)^2+
\frac1{2L} \int_0^L \bigl(\frac{\partial u}{\partial y_2}\bigr)^2,\quad \text{on}\; \{y_2=\pi\}.
\end{equation}
Thus, we see that \eqref{eq1-2} and \eqref{eq1-3} can be regarded as a perturbation
of two decoupled equations \eqref{24-24-8} and \eqref{10-21-8}.

The connection between $\eta$ and  $\Gamma_2$ is given by the 
the following problem
\begin{equation}\label{20-24-8}
\begin{cases}
\Delta \Gamma_2=0,\;\; \text{in}\; D_L,\\
\Gamma_2(y_1,0)=0, \Gamma_2(y_1,\pi)=\eta, \;\; \text{ $\Gamma_2$  is periodic in $y_1$ in $D_L$},
\end{cases}
\end{equation}
and $\Gamma_1$ is a harmonic conjugate of $\Gamma_2$ in $D_L$,
which can be chosen as
\[
\begin{split}
\Gamma_1(y_1,y_2)=& \int_{(0,L)}^{ (y_1,y_2) }\frac{\partial \Gamma_2}{\partial y_2}\,dy_1-\frac{\partial \Gamma_2}{\partial y_1}\,dy_2.
\end{split}
\]

\bigskip

\subsection{The case $\alpha = 0$.}

The Bernoulli equation is given by
\begin{equation}\label{1-15-6}
\frac12 \frac{1}{|1+\Gamma'|^2} \bigl(\frac{\partial u}{\partial y_2}\bigr)^2 +g(\pi+\Gamma_2)=C,\quad \text{on}\; \{y_2=\pi\},
\end{equation}
which can be regarded as a perturbation of the following equation:
\begin{equation}\label{n1-15-6}
\frac12  \bigl(\frac{\partial u}{\partial y_2}\bigr)^2 +g(\pi+\Gamma_2)=C.
\end{equation}
To solve \eqref{1-15-6}, we have to control the derivatives
of $\Gamma_2$. However, unlike \eqref{10-21-8},
\eqref{n1-15-6} does not provide any a priori bound for $\Gamma_2$.
To make the study of this problem simpler, we introduce a
background periodic flow $\hat \psi_0$ with a constant vorticity
to the problem.
So $\hat \psi_0$ satisfies $-\Delta \hat \psi_0 =\gamma$ in $\Omega$.
Equivalently, we want to find a function $\psi_0$ satisfying
\begin{equation}\label{2-15-6}
-\Delta \psi_0 =|1+\Gamma'|^2 \gamma, \quad \text{in}\; D_L,
\end{equation}
with $\psi_0=0$ on $y_2=0,\;\pi$, and $\psi$ periodic in $y_1$
with period $L$.

In the case $\alpha=0$, we look for a stream function
of the form $u+\psi_0$, 
and $u$ satisfies
the following equation
\begin{equation}\label{n10-24-8}
\begin{cases}
-\Delta u= |1+\Gamma'|^2\frac1{\ep^2}\sum\limits_{j=1}^k 1_{B_\delta(x_{0,j})} f_j(u+\psi_0) ,\;
\; y\in D_L,\\
u(y_1, 0)=u(y_1, \pi)=0,\;\; \text{$u$ is periodic in $y_1$ in $D_L$},
\end{cases}
\end{equation}
which is a perturbation of 

\begin{equation}\label{20-11-9}
\begin{cases}
-\Delta u= \frac1{\ep^2}\sum\limits_{j=1}^k 1_{B_\delta(x_{0,j})} f_j(u+\psi_0) ,\;
\; y\in D_L,\\
u(y_1, 0)=u(y_1, \pi)=0,\;\; \text{$u$ is periodic in $y_1$ in $D_L$}.
\end{cases}
\end{equation}

On the other hand,
the Bernoulli equation is
\begin{equation}\label{10-15-6}
\frac12 \frac{1}{|1+\Gamma'|^2} \bigl(\frac{\partial (u+\psi_0)}{\partial y_2}\bigr)^2 +g(\pi+\Gamma_2)=C,\quad \text{on}\; \{y_2=\pi\}.
\end{equation}
With this introduction of the background flow,
\eqref{10-15-6} is equivalent to
\begin{equation}\label{11-15-6}
\frac12  \Bigl((\frac{\partial u}{\partial y_2})^2 +
2\frac{\partial u}{\partial y_2}\frac{\partial \psi_0}{\partial y_2}
+(\frac{\partial \psi_0}{\partial y_2})^2\Bigr) 
\Bigl( 1-2 \frac{\partial \Gamma_2}{\partial y_2}+ O\bigl(
|\nabla \Gamma_2|^2\bigr)   \Bigr)+g(\pi+\Gamma_2)=C,\quad \text{on}\; \{y_2=\pi\}.
\end{equation}

The first main term of the solution of
\eqref{2-15-6} is
\begin{equation}\label{30-21-8}
\psi_{01}=\gamma\Bigl(\frac{\pi^2}{8} -\frac12\big( y_2-\frac\pi2\bigr)^2\Bigr),
\end{equation}
while the second term $\psi_{02}$ is the solution of
\begin{equation}\label{3-15-6}
-\Delta \psi_{02} =2\gamma \frac{\partial \Gamma_2}{\partial y_2}, \quad \text{in}\; D_L,\quad \psi_{02}(y_1,0)=\phi(y_1, \pi)=0.
\end{equation}
We can prove that
\begin{equation}\label{4-15-6}
 \frac{\partial \psi_{02}}{\partial y_2}\bigr|_{y_2=\pi}= -\gamma  \Gamma_2(y_1,\pi).
\end{equation}

Combining \eqref{30-21-8} and \eqref{4-15-6}, we find that
\[
\frac{\partial \psi_0}{\partial y_2}=-\frac{\pi\gamma} 2  -\gamma  \Gamma_2(y_1,\pi)+
O\Bigl(|\gamma| |\Gamma'|^2\Bigr).
\]
Thus, \eqref{11-15-6} becomes
\begin{equation}\label{12-15-6}
-\frac{\pi^2\gamma^2}4 \frac{\partial \Gamma_2}{\partial y_2}+\bigl(g+\frac\pi2 \gamma^2 
  \bigr)\Gamma_2-\frac{\pi \gamma}2 \frac{\partial u}{\partial y_2}+g\pi= C+\mathcal F (u,\Gamma),\quad \text{on}\; \{y_2=\pi\},
\end{equation}
where the function $\mathcal F (u,\Gamma)$ satisfies
\begin{equation}\label{0-24-8}
\mathcal F(u,\Gamma)= O\Bigl((\frac{\partial u}{\partial y_2})^2+\Gamma_2^2+
|\nabla \Gamma_2|^2\Bigr).
\end{equation}

Now we take $a=g+\frac\pi2 \gamma^2 L$ in \eqref{1-23-8} to obtain
\begin{equation}\label{30-23-8}
\begin{split}
C=&\frac1L \int_{0}^L  \frac12 \frac{1}{|1+\Gamma'|^2} \bigl(\frac{\partial u}{\partial y_2}\bigr)^2
\,dy_1-
\frac aL\int_{0}^L \Gamma_2(t, \pi)\,dt+ g\pi\\
=&-\frac{\pi^2\gamma^2}{4L} \int_{0}^L\frac{\partial \Gamma_2}{\partial y_2}+\bigl(g+\frac\pi2 \gamma^2 
  \bigr)\frac1L \int_{0}^L\Gamma_2-\frac{\pi \gamma}2 \frac1L \int_{0}^L\frac{\partial u}{\partial y_2}\\
  &-
\frac aL\int_{0}^L \Gamma_2(t, \pi)\,dt+g\pi+\frac1L \int_{0}^L\mathcal F(u,\Gamma)\\
=&-\frac{\pi^2\gamma^2}{4L} \int_{0}^L\frac{\partial \Gamma_2}{\partial y_2}
-\frac{\pi \gamma}2 \frac1L \int_{0}^L\frac{\partial u}{\partial y_2}
+g\pi+\frac1L \int_{0}^L\mathcal F(u,\Gamma).
\end{split}
\end{equation}
So we see that \eqref{12-15-6} becomes
\begin{equation}\label{62-24-8}
\begin{split}
&-\frac{\pi^2\gamma^2}4 \frac{\partial \Gamma_2}{\partial y_2}
+\frac{\pi^2\gamma^2}{4L} \int_{0}^L\frac{\partial \Gamma_2}{\partial y_2}+\bigl(g+\frac\pi2 \gamma^2 
  \bigr)\Gamma_2\\
  = &\frac{\pi \gamma}2 \frac{\partial u}{\partial y_2}
  -\frac{\pi \gamma}2 \frac1L \int_{0}^L\frac{\partial u}{\partial y_2}
+\mathcal F(u,\Gamma)-  \frac1L \int_{0}^L\mathcal F(u,\Gamma),\quad \text{on}\; \{y_2=\pi\},
  \end{split}
\end{equation}
which is a perturbation of the following problem:
\begin{equation}\label{22-21-8}
-\frac{\pi^2\gamma^2}4 \frac{\partial \Gamma_2}{\partial y_2}
+\frac{\pi^2\gamma^2}{4L} \int_{0}^L\frac{\partial \Gamma_2}{\partial y_2}+\bigl(g+\frac\pi2 \gamma^2 
  \bigr)\Gamma_2= f,\quad \text{on}\; \{y_2=\pi\},
\end{equation}
where $f= \frac{\pi \gamma}2 \frac{\partial u}{\partial y_2}
  -\frac{\pi \gamma}2 \frac1L \int_{0}^L\frac{\partial u}{\partial y_2}$.
  
 The striking difference between the  pure gravity case and the  capillary-gravity case
 is that \eqref{20-11-9} and \eqref{22-21-8} are not decoupled,
 because $\psi_0$ also depends on $\Gamma$, even if we just take the
 first two terms of $\psi_0$.

\subsection{Approximate solutions for the elliptic equation}

We construct approximate solutions for \eqref{20-11-9}. Note
that\eqref{24-24-8} can be regarded as the special case $\gamma=0$.

Recall that $f_j(u)=(u-\kappa_j)_+^p$, or $f_j(u)=-(-u-\kappa_j)_+^p$.
We consider
\begin{equation}\label{eq3-1}
-\varepsilon^2\Delta u= (u-a)_+^p,
\,\,\, \text{in}\; \mathbb{R}^2.
\end{equation}
First, we know that the following equation
\[-\Delta\phi=\phi^p,\,\,\,\phi>0,\,\,\, \phi\in H^1_0(B_1(0))\]
has a unique solution $\phi$, which is a radial function.
Moreover, it satisfies
\begin{equation}\label{0-11-10}
\int_{B_1(0)}\, \phi^{p+1}=\frac{\pi (p+1)}{2}|\phi^\prime(1)|^2,\,\,\, \int_{B_1(0)}\, \phi^{p}=2\pi|\phi^\prime(1)|.
\end{equation}
Then \eqref{eq3-1} has a solution $U_{\varepsilon,a}$,
which can be represented by
\begin{equation}\label{eq3-2}
U_{\varepsilon,a}(y)=\begin{cases}
a+(\frac{\varepsilon}{s_\varepsilon})^\frac{2}{p-1}\phi(\frac{|y|}{s_\varepsilon}),
\,\,\,  &|y|\leq s_\varepsilon,\\
\frac{a}{\ln s_\varepsilon}\ln|y|,\,\,\, &|y|>s_\varepsilon,
\end{cases}
\end{equation}
where $s_\varepsilon$ is a constant such that
$U_{\varepsilon,a}\in C^1(\mathbb{R}^2)$.
So, $s_\varepsilon$ is determined by
\begin{equation}\label{eq3-3}
(\frac{\varepsilon}{s_\varepsilon})^\frac{2}{p-1}\phi^\prime(1)=\frac{a}{\ln s_\varepsilon},
\end{equation}
which gives an expansion for $s_\varepsilon$ as follows:
\begin{equation}\label{eq3-4}
s_\varepsilon=(\frac{a}{|\phi^\prime(1)|})^{\frac{1-p}{2}}\varepsilon|\ln\varepsilon|^\frac{p-1}{2}(1+O(\frac{\ln|\ln\varepsilon|}{|\ln\varepsilon|})).
\end{equation}

%\begin{equation}\label{1-8-10}
%\begin{cases}
%-\ep^2\Delta u=  f_j(u),&\text{in}\,\,D,\\
%u=0,&\text{on}\,\,\partial D.
%\end{cases}
%\end{equation}

For each $x\in \Omega$, we set
$U_{\varepsilon,x,a}(y)=U_{\varepsilon,a}(y-x) $.
We define the projection $PU_{\varepsilon,x,a}$ of
$U_{\varepsilon,x,a}$ as the solution of the following problem
\begin{equation}\label{2-8-10}
\begin{cases}
-\ep^2 \Delta v=(U_{\varepsilon,x,a}-a)_+^p,\;\; \text{in}\;D_L,\\
v=0,\;\; \text{on $y_2=0$ or $y_2=\pi$},\\
v \; \text{is periodic in $y_1$}.
\end{cases}
\end{equation}

Denote by $G_p(y,x)$ the Green's function for $-\Delta $ in
$D_L$, satisfying the zero Dirichlet boundary condition on $y_2=0$
and $y_2=\pi$, and the periodic boundary condition in $y_1$.
Then
\begin{equation}\label{1-8-5}
G_p(y,x)=\sum_{j=-\infty}^{+\infty}G(y+jL \boldsymbol{e}_1, x)
\end{equation}
where $\boldsymbol{e}_1=(1,0)$, and $G(y,x)$ is the Green's function for
$-\Delta $ in $D=\mathbb R\times (0, \pi)$,
satisfying the zero Dirichlet boundary condition.
Thus we have
\begin{equation}\label{2-8-5}
PU_{\varepsilon,x,a}(y)= \frac1{\ep^2}
\int_{D_L} G_p(z,y)(U_{\varepsilon,x,a}(z)-a)_+^p\,dz,
\end{equation}
and
\begin{equation}\label{3-8-5}
\begin{split}
PU_{\varepsilon,x,a}(y)=&\bigl( G_p(x,y)+o(1)\bigr) \frac1{\ep^2}
\int_{D_L} (U_{\varepsilon,x,a}(z)-a)_+^p\,dz\\
=& -\frac{2\pi a}{\ln s_\ep}\bigl( G_p(x,y)+o(1)\bigr)\quad y\in D_L\setminus B_\delta(x).
\end{split}
\end{equation}
See the computations in \eqref{10-5-8}.

Let $\mathbf{x}=(x_1,\cdots,x_k)$, where $x_j\in D_L$,
$j=1,\cdots,k$, with $x_i\ne x_j$, $i\ne j$.
The approximate solution for \eqref{20-11-9} is defined by
\begin{equation}\label{eq3-7} 
\mathcal{U}_{\varepsilon,\mathbf{x},\mathbf{a}}=\sum\limits_{j=1}^k
(-1)^{\tau_j}PU_{\varepsilon,x_j,a_j},
\end{equation}
where $\tau_j=0$ or 1, and $a_j$ is some constant close to
$\kappa_j-(-1)^{\tau_j} \psi_{01}(x_{0,j})$,
$j=1,\cdots, k$.
Here $\psi_{01}(y) =\gamma\Bigl(\frac{\pi^2}{8} -\frac12\big( y_2-\frac\pi2\bigr)^2\Bigr)$,
which is the main term of the background flow $\psi_0$. We take $\gamma\ne 0$, such that

\[
\kappa_j-(-1)^{\tau_j} \psi_{01}(x_{0,j})>0.
\]

Now let us discuss the location of the concentration point $x_{0,j}$,
$j=1,\cdots, k$.
Suppose that $u_\ep$ is a solution of \eqref{20-11-9},
satisfying $\text{spt} u_\ep 
\subset\subset\cup_{j=1}^k B_\delta(x_{0, j})$.
Then we have the following Pohozaev identities
\begin{equation}\label{6-26-8}
\begin{split}
 \int_{\partial B_\delta(x_m)} \frac{\partial  u_{\varepsilon} }{\partial \nu}
\frac{\partial  u_{\varepsilon} }{\partial y_l}-\frac12
\int_{\partial B_\delta(x_m)} |\nabla u_{\varepsilon}|^2\nu_j-\frac1{\ep^2}\int_{ B_\delta(x_m)}
f_m(u_\ep+\psi_0)\frac{\partial \psi_0}{\partial y_l}
=0.
\end{split}
\end{equation}
Thus, by \eqref{eq3-7} and \eqref{3-8-5}, we can derive that
$\mathbf{x}_0=(x_{0,1},\cdots,x_{0,k})$
is a critical point of the following function
\begin{equation}\label{7-26-8}
\begin{split}
\tilde{\mathcal K_p}(\mathbf{x})=
\sum_{j=1}^k \hat\kappa_j^2 R_p(x_j) - \sum_{i\ne j} (-1)^{\tau_i+\tau_j}\hat\kappa_i\hat\kappa_j G_p(x_i, x_j)-
2\sum_{j=1}^k (-1)^{\tau_j} \hat\kappa_j   \psi_{01}(x_j).
\end{split}
\end{equation}
where $\mathbf{x}=(x_1,\cdots,x_k)$,  $\hat \kappa_j=   \kappa_j-(-1)^{\tau_j} \psi_{01}(x_{0,j})$,
$R_p(x)=H_p(x,x)$ is the Robin function,
and $H_p(x,x)$ is the regular part of $G_p(y, x)$.
Hence, a necessary condition for the solvability of \eqref{20-11-9}
is that $\tilde{\mathcal K_p}(\mathbf{x})$ has a critical point.

\bigskip
\subsection{Critical points of Kirchhoff-Routh functions}

In this subsection, we briefly discuss the 
Kirchhoff-Routh type function
$\tilde {\mathcal K}_p$.
The Green's function for $-\Delta $ in $D$,
satisfying the zero Dirichlet boundary condition, is given by
\begin{equation}\label{1-9-10}
\begin{split}
G(y,x)=\frac{1}{4\pi}\ln (1+\frac{2 \sin x_2\sin y_2}{\cosh(y_1-x_1)-\cos (y_2-x_2)}).
\end{split}
\end{equation}
Hence, $G_p(y,x)$, the Green's function for $-\Delta $ in
$D_L$, satisfying the zero Dirichlet boundary condition on $y_2=0$
and $y_2=\pi$, and the periodic boundary condition in $y_1$,
can be written as
\begin{equation}\label{10-26-8}
G_p(y,x)=\frac{1}{4\pi}\sum_{j=-\infty}^{+\infty}\ln (1+\frac{2 \sin x_2\sin y_2}{\cosh(y_1+jL-x_1)-\cos (y_2-x_2)}).
\end{equation}
Then, the regular part of $G_p(y,x)$ is
\begin{equation}\label{11-26-8}
H_p(y,x)=\frac1{2\pi}\ln \frac1{|y-x|}-\frac{1}{4\pi}\sum_{j=-\infty}^{+\infty}\ln (1+\frac{2 \sin x_2\sin y_2}{\cosh(y_1+jL-x_1)-\cos (y_2-x_2)}),
\end{equation}
and hence
\begin{equation}\label{12-26-8}
R_p(x)=-\frac{1}{2\pi}\ln (2\sin x_2)-\frac{1}{4\pi}\sum_{j\ne 0}\ln (1+\frac{2 \sin^2 x_2}{\cosh(jL)-1}).
\end{equation}

Our first observation on the function
$\tilde {\mathcal K}_p$ is that its partial derivatives 
 with respect to $x_{j2}$
are zero at $(x_{j1},\frac\pi2)$.
As a result, we will try to find critical points of the form
$(x_{j1},\frac\pi2)$ for some $x_{j1}$, $j=1,\cdots,k$.
Our objective is to construct solutions which, up to a translation,
are not even in $y_1$.
This can be achieved if  $\tilde {\mathcal K}_p$
have critical points of the form $(x_{j1},\frac\pi2)$,
such that the set $\{x_{11},\cdots,x_{k1}\}$ has no symmetry.

Another important observation is that
 $\tilde {\mathcal K}_p(\mathbf x+ t e )=\tilde {\mathcal K}_p(\mathbf x )$,
where $e= ( (1,0),\cdots,(1, 0))$.
As a result, if $\tilde {\mathcal K}_p(\mathbf x )$ has a critical point,
then it is not isolated and thus is degenerate.
Generally, a degenerate critical point is not stable under all
$C^1$ perturbation.
So, to construct a solution to the corresponding elliptic problems,
we must show that the error term obtained in the reduction procedure
is translation invariant in $y_1$.

\bigskip

In summary, we construct the approximate solutions for
\eqref{eq1-2}-\eqref{eq1-3} and \eqref{1-23-8} as follows.

\medskip

{\bf The case $\alpha>0$}. In this case, no background flow is needed.
The approximate solution for $u$ is
$\mathcal{U}_{\varepsilon,\mathbf{x},\mathbf{a}}$ with $a_j$
close to $\kappa_j$ and $x_j$ close to $x_{0,j}$,
while the approximate solution for $\eta(t)=\Gamma_2(t,\pi)$
is the solution of
\begin{equation}\label{80-24-8}
-\alpha^2 \eta'' +g\eta=-\frac12  \bigl(\frac{\partial \mathcal{U}_{\varepsilon,\mathbf{x},\mathbf{a}}}{\partial y_2}\bigr)^2+
\frac1{2L} \int_0^L \bigl(\frac{\partial \mathcal{U}_{\varepsilon,\mathbf{x},\mathbf{a}}}{\partial y_2}\bigr)^2,\quad \text{on}\; \{y_2=\pi\},
\end{equation}
satisfying the periodic boundary condition at $t=0, \; L$.
We aim to prove that if $\tilde {\mathcal K}_p$ (with $\gamma=0$) has a critical point
$\mathbf{x}_0=(x_{0,1},\cdots,x_{0,k})$,
then the system \eqref{10-24-8} and \eqref{61-21-8}
has a solution $(u_\ep, \eta_\ep)$, such that $u_\ep$ concentrates at $\mathbf{x}_0$
as $\ep\to 0$.

\medskip

{\bf The case $\alpha=0$}. In this case, 
 $\Gamma_2$ is determined by the following problem
\begin{equation}\label{21-24-8}
\begin{cases}
\Delta \Gamma_2=0,\;\; \text{in}\; D_L,\\
\Gamma_2(y_1,0)=0,  \;\; \text{  $\Gamma_2(y_1,\pi)$ satisfies \eqref{62-24-8}, }\\
\text{ $\Gamma_2$  is periodic in $y_1$ in $D_L$},
\end{cases}
\end{equation}

The approximate solution for $u$ is
$\mathcal{U}_{\varepsilon,\mathbf{x},\mathbf{a}}$ with
$a_j$ close to $\kappa_j-(-1)^{\tau_j}\psi_{01}(x_{0,j})$,
$j=1,\cdots, k$,
where $\psi_{01}(y) =\gamma\Bigl(\frac{\pi^2}{8} -\frac12\big( y_2-\frac\pi2\bigr)^2\Bigr)$,
and $x_j$ close to $x_{0,j}$,
while the approximate solution for  $\Gamma_2$ is the solution of
\begin{equation}\label{81-24-8}
\begin{cases}
\Delta \Gamma_2=0,\;\; \text{in}\; D_L,\\
\Gamma_2(y_1,0)=0,  \\
-\frac{\pi^2\gamma^2}4 \frac{\partial \Gamma_2}{\partial y_2}
+\frac{\pi^2\gamma^2}{4L} \int_{0}^L\frac{\partial \Gamma_2}{\partial y_2}+\bigl(g+\frac\pi2 \gamma^2 
  \bigr)\Gamma_2= \frac{\pi \gamma}2 \frac{\partial \mathcal{U}_{\varepsilon,\mathbf{x},\mathbf{a}}}{\partial y_2}
  -\frac{\pi \gamma}2 \frac1L \int_{0}^L\frac{\partial \mathcal{U}_{\varepsilon,\mathbf{x},\mathbf{a}}}{\partial y_2},\quad \text{on}\; \{y_2=\pi\},\\
\text{ $\Gamma_2$  is periodic in $y_1$ in $D_L$}.
\end{cases}
\end{equation}
We aim to prove that if $\tilde{\mathcal K}_p$ has a critical point
$\mathbf{x}_0=(x_{0,1},\cdots,x_{0,k})$,
then the system \eqref{n10-24-8} and \eqref{21-24-8}
has a solution $(u_\ep, \Gamma_\ep)$, such that $u_\ep$ concentrates at $\mathbf{x}_0$
as $\ep\to 0$.

\medskip

\begin{remark}

Before ending this section, we make the following remarks.

\begin{itemize}
\item In the case $\alpha>0$,
if we ignore the perturbation term $\mathcal F_1(u, \Gamma)$,
\eqref{61-21-8} becomes a good second-order ODE.
On the other hand, if we ignore $\Gamma$ in \eqref{10-24-8},
it becomes a pure nonlinear elliptic problem.
So, in the case $\alpha>0$, the problem is a weakly coupled 
system.

\item The case $\alpha=0$ is quite different.
It follows from \eqref{22-21-8} that the free boundary equation
involves the Dirichlet-Neumann operator $\frac{\partial \Gamma_2}
{\partial y_2}$.
In this paper, we do not deal with this operator directly. Instead,
we investigate \eqref{21-24-8}, which is an elliptic
problem, satisfying a nonlocal mix boundary condition.
System \eqref{n10-24-8} and \eqref{21-24-8} can not be decoupled 
even if we ignore all the perturbation terms.

\item

We note that the asymmetry of water waves follows from the asymmetry of the  non-degenerate  critical point
 of $\tilde{\mathcal K}_p(\mathbf{x})$ given by \eqref{7-26-8}. Indeed, such asymmetry 
 results in the asymmetry of the main term $\mathcal{U}_{\varepsilon,\mathbf{x},\mathbf{a}}$
 of the stream function $\psi$.  Moreover, it follows from \eqref{81-24-8} (or \eqref{80-24-8})
that the free surface $y_2=\eta(y_1)$ is also asymmetric.

\end{itemize}
\end{remark}

\section{estimate of the approximate solutions}

 In this section, we estimate the approximate solution $\mathcal{U}_{\varepsilon,\mathbf{x},\mathbf{a}}$ defined in \eqref{eq3-7}.
First, we estimate $PU_{\varepsilon,x,a}(y)$.

 \begin{lemma}\label{l1-8-5}
 We have

 \[
 PU_{\varepsilon,x,a}(y)= U_{\varepsilon,x,a}(y)+\frac{2a\pi}{\ln s_\varepsilon}H(y,x) -\frac{2\pi a}{\ln s_{\varepsilon}}\sum_{j\ne 0} G(x+jL \boldsymbol{e}_1, y)
  + O(  \frac{ s_\ep^2}{|\ln s_{\varepsilon}|} \sum_{j\ne 0}e^{-|L \boldsymbol{e}_1+x-y|  } ),
 \]
 where $H(y,x)$ is the regular part of the Green's function $G(y,x)$ for $-\Delta $ in
$D$.
 \end{lemma}

 \begin{proof}
 In view of $(U_{\varepsilon,x,a}(z)-a)_+=0$ outside
$B_{s_\ep}(x)$, from \eqref{2-8-5}, we find that

\begin{equation}\label{5-8-5}
\begin{split}
PU_{\varepsilon,x,a}(y)=&  \frac1{\ep^2}
\int_{D} G(z,y) (U_{\varepsilon,x,a}(z)-a)_+^p\,dz\\
&+
\frac1{\ep^2}\sum_{j\ne 0} \int_{D} G(z+jL \boldsymbol{e}_1, y)
  (U_{\varepsilon,x,a}(z)-a)_+^p\,dz.
\end{split}
\end{equation}
 Let

 \[
 \bar PU_{\varepsilon,x,a}=\frac1{\ep^2}
\int_{D} G(z,y) (U_{\varepsilon,x,a}(z)-a)_+^p\,dz.
 \]
Then $\bar PU_{\varepsilon,x,a}$ is the solution of the following problem

\begin{equation}\label{nn2-8-10}
\begin{cases}
-\ep^2 \Delta v=(U_{\varepsilon,x,a}-a)_+^p,& \text{in}\;D,\\
v=0,& \text{on}\;\partial D.
\end{cases}
\end{equation}
Let $\varphi_{\varepsilon,x,a}=U_{\varepsilon,x,a}-\bar PU_{\varepsilon,x,a}$. We
have

\[ \begin{cases}
 \Delta \varphi_{\varepsilon,x,a}=0,& \text{in}\;D,\\
\varphi_{\varepsilon,x,a}=U_{\varepsilon,x,a},& \text{on}\;\partial D.
\end{cases}
\]
Since $U_{\varepsilon,x,a}= \frac{a}{\ln s_\varepsilon}\ln|y-x| $, we conclude that

\[
\varphi_{\varepsilon,x,a}=-\frac{2a\pi}{\ln s_\varepsilon}H(y,x).
\]
  So we see

\begin{equation}\label{eq3-5}
\bar PU_{\varepsilon,x,a}(y)=U_{\varepsilon,x,a}(y)+\frac{2a\pi}{\ln s_\varepsilon}H(y,x).
\end{equation}

For $j\ne 0$, it holds
that

\begin{equation}\label{6-8-5}
\begin{split}
&
 \int_{D} G(z+jL \boldsymbol{e}_1, y)
  (U_{\varepsilon,x,a}(z)-a)_+^p\,dz\\
  =&\int_{B_{s_\ep  } (x)} G(z+jL \boldsymbol{e}_1, y)
  (U_{\varepsilon,x,a}(z)-a)_+^p\,dz\\
  =&\int_{B_{s_\ep  } (x)}\bigl[ G(x+jL \boldsymbol{e}_1, y)+\langle
  \nabla G(x+jL \boldsymbol{e}_1, y), z-x\rangle \\
  &\qquad\quad+ O( |z-x|^2 e^{-|L \boldsymbol{e}_1+x-y|  }  )
  \bigr]
  (U_{\varepsilon,x,a}(z)-a)_+^p\,dz\\
  =&\int_{B_{s_\ep  } (x)}\bigl[ G(x+jL \boldsymbol{e}_1, y)+ O( s_\ep^2 e^{-|L \boldsymbol{e}_1+x-y|  }  )
  \bigr]
  (U_{\varepsilon,x,a}(z)-a)_+^p\,dz
\end{split}
\end{equation}

On the other hand, we have

\begin{equation}\label{10-5-8}
\begin{split}
&\int_{B_{x_s}(x)}\,( U_{\varepsilon,x,a}-a)_+^{p}\\
=&\int_{B_{s_{\varepsilon}}(x)}
\Bigl[(\frac{\varepsilon}{s_{\varepsilon}})^{\frac{2}{p-1}}
\phi(\frac{|y-x|}{s_{\varepsilon}})\Bigr]^{p}
=
(\frac{\varepsilon}{s_{\varepsilon}})^{\frac{2p}{p-1}}
s_{\varepsilon}^2 \int_{B_1(0)}
\phi^{p}\\
=&-2\pi\varepsilon^2\frac{a}{\ln s_{\varepsilon}}.
\end{split}
\end{equation}

Inserting \eqref{10-5-8} into \eqref{6-8-5} yields

\begin{equation}\label{11-8-5}
\begin{split}
&
 \int_{D} G(z+jL \boldsymbol{e}_1, y)
  (U_{\varepsilon,x,a}(z)-a)_+^p\,dz\\
  =& -2\pi\varepsilon^2\frac{a}{\ln s_{\varepsilon}}G(x+jL \boldsymbol{e}_1, y)
  + O( s_\ep^2 e^{-|L \boldsymbol{e}_1+x-y|  } \varepsilon^2\frac{1}{\ln s_{\varepsilon}} ).
\end{split}
\end{equation}

\end{proof}

We now turn to the construction of the approximate solutions for the elliptic problems
\eqref{20-11-9} and \eqref{24-24-8}. We just study \eqref{20-11-9}, because \eqref{24-24-8}
can be regarded as a special case of \eqref{20-11-9} with $\gamma=0$.

First, we give a better approximation of the background flow $\psi_0$.
As in \eqref{30-21-8}, the first approximation is given by

\begin{equation}\label{1-27-8}
\psi_{01}(y) =\gamma\Bigl(\frac{\pi^2}{8} -\frac12\big( y_2-\frac\pi2\bigr)^2\Bigr).
\end{equation}
On the other hand, the second approximation is given by the solution of \eqref{3-15-6},
which depends on the unknown free boundary $\Gamma_2(y_1,\pi)$.

 Denote by $\Gamma_0$ the solution of the following problem
 
 \begin{equation}\label{8-29-8}
\begin{cases}
\Delta  \Gamma_0=0,\;\; \text{in}\; D_L,\\
 \Gamma_0(y_1,0)=0,  \\
-\frac{\pi^2\gamma^2}4 \frac{\partial  \Gamma_0}{\partial y_2}
+\frac{\pi^2\gamma^2}{4L} \int_{0}^L\frac{\partial \Gamma_0}{\partial y_2}+\bigl(g+\frac\pi2 \gamma^2 
  \bigr)\Gamma_0= \frac{\pi \gamma}2 \frac{\partial \mathcal{U}_{\varepsilon,\mathbf{x},\mathbf{\hat a}}}{\partial y_2}
  -\frac{\pi \gamma}2 \frac1L \int_{0}^L\frac{\partial \mathcal{U}_{\varepsilon,\mathbf{x},\mathbf{\hat a}}}{\partial y_2},\quad \text{on}\; \{y_2=\pi\},\\
\text{ $\Gamma_0$  is periodic in $y_1$ in $D_L$}.
\end{cases}
\end{equation}
  with 
 
 \begin{equation}\label{2-29-8}
\hat  a_j=\kappa_j-(-1)^{\tau_j}\psi_{01}(x_j).
\end{equation}

For any function $\zeta$, which is periodic in $y_1$,  we denote by
 $\psi_{02,\zeta}$  the solution of

 \begin{equation}\label{100-29-8}
 \begin{cases}
-\Delta \psi_{02,\zeta} =2 \gamma\partial_{y_2}\xi, \quad \text{in}\; D_L=:(0, L)\times (0, \pi),\\
\psi_{02,\zeta}=0,\;\text{ on $y_2=0,\;\pi$, \;\  and $\psi_{02,\zeta}$ is periodic in $y_1$ with period $L$}.
\end{cases}
\end{equation}
 
We take

 \begin{equation}\label{20-2-9}
\tilde \psi=   \psi_{01}+\psi_{02,\Gamma_0}.
\end{equation}
Then $\tilde \psi$ is a better approximate of the background flow $\psi_0$.

By \eqref{eq3-7}, we have
\begin{equation}\label{eq3-9}
\begin{split}
-&\ep^2\Delta \mathcal{U}_{\varepsilon,\mathbf{x},\mathbf{a}}-
\sum\limits_{j=1}^k
 1_{B_\delta(x_{0,j})} f_j(\mathcal{U}_{\varepsilon,\mathbf{x},
\mathbf{a}}+\tilde  \psi)
\\
=&
\sum\limits_{j=1}^k (-1)^{\tau_j} (U_{\varepsilon,x_j,a_j}
-a_j)_+^p
-
\sum\limits_{j=1}^k1_{B_\delta(x_{0,j})} f_j(\mathcal{U}_{\varepsilon,\mathbf{x},
\mathbf{a}}+\tilde  \psi).
\end{split}
\end{equation}
We now choose $a_j$ suitably, such that  the right hand side of \eqref{eq3-9}
is small. 

Fix $\delta>0$  small.
For $y\in B_{\delta}(x_j)$ with $\tau_j=0$, we have
$f_j(u)= (u-\kappa_j)_+^p$.  From \eqref{eq3-5},
we see that for  $|y-x|\ge \delta>0$,

\[
U_{\varepsilon,x,a}(y)+\frac{2a\pi}{\ln s_\varepsilon}H(y,x)= -\frac{2a\pi}{\ln s_\varepsilon} G(y,x),
\]
which, in view of Lemma~\ref{l1-8-5}, gives

\begin{equation}\label{12-8-5}
\begin{split}
PU_{\varepsilon,x,a}(y)=&-\frac{2a\pi}{\ln s_\varepsilon} G(y,x)  -\frac{2\pi a_{\varepsilon}}{\ln s_{\varepsilon}}\sum_{j\ne 0} G(x+jL \boldsymbol{e}_1, y)
  + O(  \frac{ s_\ep^2}{\ln s_{\varepsilon}} \sum_{j\ne 0}e^{-|L \boldsymbol{e}_1+x-y|  } )\\
 =&-\frac{2a\pi}{\ln s_\varepsilon} G_p(y,x)  + O(  \frac{ s_\ep^2}{\ln s_{\varepsilon}} \sum_{j\ne 0}e^{-|L \boldsymbol{e}_1+x-y|  } ), \quad y\notin B_\delta(x).
\end{split}
\end{equation}

From Lemma~\ref{l1-8-5} and \eqref{12-8-5}, we find

\begin{equation}\label{eq3-13}
\begin{split}
&\mathcal{U}_{\varepsilon,\mathbf{x},
\mathbf{a}}(y)+\tilde  \psi(y)
-\kappa_j
\\
=&U_{\varepsilon,x_j,a_{\varepsilon,j}}(y)-\kappa_j+\tilde  \psi(y)+\frac{2\pi a_{\varepsilon,j}}{\ln s_{\varepsilon,j}}H_p(y,x_j)
-\sum\limits_{i=1,i\neq j}^k\frac{2\pi a_{\varepsilon,i}(-1)^{\tau_i}}{\ln s_{\varepsilon,i}}G_p(y,x_i)
\\
=&U_{\varepsilon,x_j,a_{\varepsilon,j}}(y)-\kappa_j+\tilde  \psi(x_j)+\frac{2\pi a_{\varepsilon,j}}{\ln s_{\varepsilon,j}}H_p(x_j,x_j)-\sum\limits_{i=1,i\neq j}^k\frac{2\pi a_{\varepsilon,i}(-1)^{\tau_i}}{\ln s_{\varepsilon,i}}G_p(x_j,x_i)\\
&+\langle \nabla \tilde  \psi(x_j),y-x_j\rangle+\frac{2\pi a_{\varepsilon,j}}{\ln s_{\varepsilon,j}}\langle \nabla H_p(x_j,x_j),y-x_j\rangle
\\
&-\sum\limits_{i=1,i\neq j}^k\frac{2\pi a_{\varepsilon,i}(-1)^{\tau_i}}{\ln s_{\varepsilon,i}}\langle \nabla G_p(x_j,x_i),y-x_j\rangle
+O(\frac{s_{\varepsilon,j}^2}{|\ln\varepsilon|}).
\end{split}
\end{equation}

On the other hand, for $y\in B_{\delta}(x_j)$ with $\tau_j=1$, we have
$f_j(u)= -(-u-\kappa_j)_+^p$, and

\begin{equation}\label{nneq3-13}
\begin{split}
&-(\mathcal{U}_{\varepsilon,\mathbf{x},
\mathbf{a}}(y)+\tilde  \psi(y))
-\kappa_j
\\
=&U_{\varepsilon,x_j,a_{\varepsilon,j}}(y)-\kappa_j-\tilde  \psi(x_j)+\frac{2\pi a_{\varepsilon,j}}{\ln s_{\varepsilon,j}}H_p(x_j,x_j)+\sum\limits_{i=1,i\neq j}^k\frac{2\pi a_{\varepsilon,i}(-1)^{\tau_i}}{\ln s_{\varepsilon,i}}G_p(x_j,x_i)\\
&-\langle \nabla \tilde  \psi(x_j),y-x_j\rangle+\frac{2\pi a_{\varepsilon,j}}{\ln s_{\varepsilon,j}}\langle \nabla H_p(x_j,x_j),y-x_j\rangle
\\
&+\sum\limits_{i=1,i\neq j}^k\frac{2\pi a_{\varepsilon,i}(-1)^{\tau_i}}{\ln s_{\varepsilon,i}}\langle \nabla G_p(x_j,x_i),y-x_j\rangle
+O(\frac{s_{\varepsilon,j}^2}{|\ln\varepsilon|}).
\end{split}
\end{equation}

So we see $a_{\varepsilon,j}$ and $s_{\varepsilon,j}$ should be determined by the following relations

\begin{equation}\label{nneq3-10}
a_{\varepsilon,j}=\kappa_j-(-1)^{\tau_j}\tilde \psi(x_j)-\frac{2\pi a_{\varepsilon,j}}{\ln s_{\varepsilon,j}} R_p(x_j)+\sum\limits_{i=1,i\neq j}^k\frac{2\pi a_{\varepsilon,i}(-1)^{\tau_i+\tau_j}}{\ln s_{\varepsilon,i}}G_p(x_j,x_i),
\end{equation}
and
\begin{equation}\label{eq3-11}
(\frac{\varepsilon}{s_{\varepsilon,j}})^\frac{2}{p-1}\phi^\prime(1)=\frac{a_{\varepsilon,j}}{\ln s_{\varepsilon,j}}.
\end{equation}
 We can use the implicit function theorem to  prove that for $\varepsilon>0$ small, there exist $(a_{\varepsilon,1}(\mathbf{x}),\cdots,a_{\varepsilon,k}(\mathbf{x}))$ and $(s_{\varepsilon,1}(\mathbf{x}),\cdots,s_{\varepsilon,k}(\mathbf{x}))$ such that
 \eqref{nneq3-10}
and \eqref{eq3-11} holds.

Note that $a_{\varepsilon,j}$ and $s_{\varepsilon,j}$ depend on $\mathbf{x}$.
For simplicity, in the rest of this paper, we will use 
$a_{\varepsilon,j}$ and $s_{\varepsilon,j}$ to denote  $a_{\varepsilon,j}(\mathbf{x})$ and $s_{\varepsilon,j}(\mathbf{x})$, respectively. With this choice of
$a_{\varepsilon,j}(\mathbf{x})$, it holds

\begin{equation}\label{eq3-10}
\mathcal{U}_{\varepsilon,\mathbf{x},\mathbf{a}}+\tilde  \psi-\kappa_j=
U_{\varepsilon,x_j,a_{\varepsilon,j}}-a_{\varepsilon,j}
+O(\frac{s_{\varepsilon,j}}{|\ln\varepsilon|}),\quad y\in B_{s_{\varepsilon,j}}(x_j).
\end{equation}

We can also check that
\begin{equation}\label{eq3-14}
\begin{split}
\frac{\partial a_{\varepsilon,j}}{\partial x_{ih}}=O(\frac{1}{|\ln\varepsilon|}), \quad \frac{1}{\ln s_{\varepsilon,j}}=\frac{1}{A_{\varepsilon,j}}+O(\frac{\ln|\ln\varepsilon|}{|\ln\varepsilon|^3}),
\end{split}
\end{equation}
where $A_{\varepsilon,j}=\ln\varepsilon+\frac{p-1}{2}\ln|\ln\varepsilon|+\frac{1-p}{2}\ln\frac{\kappa_j}{|\phi^\prime(1)|}$.

 By  direct computations, we obtain the following expansions which will be used in the next sections. For $j=1,\cdots,k$, and $h=1,2$,
\begin{equation}\label{eq3-15}
\begin{array}{ll}
 \displaystyle\frac{\partial U_{\varepsilon,x_j,a_{\varepsilon,j}}(y)}{\partial x_{jh}}
 =\left\{
 \begin{array}{ll}
\displaystyle\frac{a_{\varepsilon,j}}{\phi^\prime(1)\ln s_{\varepsilon,j}}\phi^\prime(\frac{|y-x_j|}{s_{\varepsilon,j}})
\frac{x_{j,h}-y_h}{|y-x_j|}\frac{1}{s_{\varepsilon,j}}&+O\left(\frac{1}{|\ln \varepsilon|}\right),\\
 \displaystyle &y\in B_{s_{\varepsilon,j}}(x_j),\\
\displaystyle
\frac{a_{\varepsilon,j}}{\ln s_{\varepsilon,j}}\frac{x_{jh}-y_h}{|y-x_j|^2}
+O\left(\frac{\ln|y-x_j|}{|\ln \varepsilon|^2}\right),
~~&y\in D_L\setminus B_{s_{\varepsilon,j}}(x_j).
\end{array}
\right.
\end{array}
\end{equation}

\begin{remark}\label{re1-15-5}

 From \eqref{2-8-5}, we have

\begin{equation}\label{22-10-5}
PU_{\varepsilon,x,a}(y)= PU_{\varepsilon,x-c\boldsymbol{e}_1,a}(y-c \boldsymbol{e}_1),
\end{equation}
where $\boldsymbol{e}_1=(1, 0)$. This gives

\begin{equation}\label{220-10-5}
\mathcal{U}_{\varepsilon,\mathbf{x},\mathbf{a}}(y)=\mathcal{U}_{\varepsilon,\mathbf{\hat x},\mathbf{a}}(y-c \boldsymbol{e}_1),
\end{equation}
where $\mathbf{\hat x}= \mathbf{x}-(c\boldsymbol{e}_1,\cdots,c\boldsymbol{e}_1)$.

\end{remark}

\section{The  linear problems}

In this section, we study the linear operators for the elliptic
equation and the free boundary equations.

\subsection{Case $\alpha=0$}  We will look for solution of the form
$u=\mathcal{U}_{\varepsilon,\mathbf{x},\mathbf{a}}+\omega$ and $\Gamma_2=\Gamma_0+\zeta
$, where 
 $\Gamma_0$ is the solution of
\eqref{8-29-8}. Note that the background flow $\psi_0=\psi_{0,\Gamma_2}$ 
also depends on $\Gamma_2$. We have

\[
\psi_{0,\Gamma_2} =\tilde \psi+\psi_{02,\zeta}+O\Bigl(|\nabla \Gamma_0|^2 +
|\nabla \zeta|^2  \Bigr)
\]
where  $\tilde \psi$  and $\psi_{02,\zeta}$ are defined in \eqref{20-2-9} and \eqref{100-29-8} respectively. 

 Then we need to study the following  linea problem 
 
\begin{equation}\label{eq4-1}
\begin{split}
L_{\varepsilon,\mathbf{x}}(\omega,\zeta):=-\ep^2\Delta \omega-\sum\limits_{j=1}^k  1_{B_\delta(x_{0,j})} \Bigl[&f'_j(\mathcal{U}_{\varepsilon,\mathbf{x},\mathbf{a}}+\tilde \psi)(\omega+\psi_{02,\zeta}
)\\
&+2
f_j(\mathcal{U}_{\varepsilon,\mathbf{x},\mathbf{a}}+\tilde \psi)\frac{\partial \zeta}{\partial y_2}
\Bigr]= f,
\end{split}
\end{equation}
and  

\begin{equation}\label{6-31-8}
\begin{cases}
\Delta \zeta=0,\quad \text{in}\; D_L;\\
\text{ $\zeta$ is   periodic in $y_1$, $\phi=0$ on $y_2=0,\;\pi$ };\\
-\frac{\pi^2\gamma^2}4 \frac{\partial \zeta}{\partial y_2}
+\frac{\pi^2\gamma^2}{4L} \int_{0}^L\frac{\partial \zeta}{\partial y_2}+\bigl(g+\frac\pi2 \gamma^2 L
  \bigr)\zeta-
  \frac{\pi \gamma}2 \frac{\partial \omega}{\partial y_2}
  +\frac{\pi \gamma}2 \frac1L \int_{0}^L\frac{\partial \omega}{\partial y_2}=f_1,\quad y_2=\pi.
  \end{cases}
\end{equation}
\bigskip

Let
\begin{equation}\label{eq4-9}
V_{\varepsilon,\mathbf{x},j,h}=\frac{\partial PU_{\varepsilon,x_j,a_{\varepsilon,j}}}{\partial x_{jh}},\quad
Z_{\varepsilon,\mathbf{x},j,h}=\Delta V_{\varepsilon,\mathbf{x},j,h},\,\,\, j=1,\cdots,k,\,\,\,h=1,2.
\end{equation}
Set
\begin{equation}\label{eq4-10}
\begin{split}
E_{\varepsilon,\mathbf{x}}=\big\{\omega:\,&\omega\in W^{2,\infty}(D_L), \;
\omega(y_1,0)=\omega(y_1,\pi)=0,\;
\text{$\omega(0,y_2)$ is periodic in $y_1$},
\\
&\int_{D_L}\, Z_{\varepsilon,\mathbf{x},j,h}\omega=0,\,j=1,\cdots, k,\, h=1,2\big\},
\end{split}
\end{equation}
and
\begin{equation}\label{eq4-11}
F_{\varepsilon,\mathbf{x}}=\big\{\omega:\,\omega\in L^\infty(D_L), \int_{D_L}\, V_{\varepsilon,\mathbf{x},j,h}\omega=0,\,j=1,\cdots, k,\, h=1,2\big\}.
\end{equation}

For any $u\in L^\infty(D_L)$, we define the projection operator to $F_{\varepsilon,\mathbf{x}}$ as follows
\begin{equation}\label{eq4-12}
Q_{\varepsilon,\mathbf{x}}u=u-\sum\limits_{j=1}^k
\sum\limits_{h=1}^2\varepsilon^2b_{jh}Z_{\varepsilon,\mathbf{x},j,h},
\end{equation}
where $b_{j1}$ and $b_{j2}$ are chosen so that $Q_{\varepsilon,\mathbf{x}}u\in F_{\varepsilon,\mathbf{x}}$.  Note that
$b_{jh}$ is determined by the following equations

\[
\int_{D_L}\, u V_{\varepsilon,\mathbf{x},i,l}=
\sum\limits_{j=1}^k\sum\limits_{h=1}^2\varepsilon^2b_{jh}
\int_{D_L}\,Z_{\varepsilon,\mathbf{x},j,h}V_{\varepsilon,\mathbf{x},i,l},\quad i=1,\cdots, k,\,\,\, l=1,2.
\]
In view of

\begin{equation}\label{eq4-13}
\begin{split}
\varepsilon^2\int_{D_L}\,Z_{\varepsilon,\mathbf{x},j,h}V_{\varepsilon,\mathbf{x},
i,l}=&p\int_{D_L}\, (U_{\varepsilon,x_j,a_{\varepsilon,j}}-a_{\varepsilon,j})_+^{p-1}(\frac{\partial U_{\varepsilon,x_j,a_{\varepsilon,j}}}{\partial x_{jh}}-\frac{\partial a_{\varepsilon,j}}{\partial x_{jh}})V_{\varepsilon,\mathbf{x},i,l}
\\
=&\delta_{ij}\delta_{hl}\frac{c}{|\ln\varepsilon|^{p-1}}+O(\frac{s_{\varepsilon,1}}{|\ln \varepsilon|^{p-1}}),
\end{split}
\end{equation}
where $c>0$ is a constant, $\delta_{ij}=1$ if $i=j$, $\delta_{ij}=0$ if $i\neq j$,   and

\begin{equation}\label{eq4-14}
\begin{split}
|\int_{D_L}\, uV_{\varepsilon,\mathbf{x},i,l}|=& |\int_{D_L}\,\frac{\partial PU_{\varepsilon,x_i,a_{\varepsilon,i}}}{\partial x_{il}} u|
\\
\leq& |\int_{B_{s_{\varepsilon,i}}(x_i)}
\,\frac{\partial U_{\varepsilon,x_i,a_{\varepsilon,i}}}{\partial x_{il}} u|
+\frac{C}{|\ln s_{\varepsilon,i}|}\|u \|_{L^\infty(D)}
\\
\le &  \frac{C}{|\ln s_{\varepsilon,i}|}\|u \|_{L^\infty(D)},
\end{split}
\end{equation}
 we find that

\[
|b_{jh  }|\le C|\ln s_{\varepsilon,i}|^{p-2}\|u \|_{L^\infty(D_L)}.
\]
Thus, we can deduce that $Q_{\varepsilon,\mathbf{x}}$ is a bounded linear
operator from $L^\infty(D_L)$ to $F_{\varepsilon,\mathbf{x}}$.

\begin{remark}\label{re1-1-9}
If $u=0$ in  $  D_L\setminus \cup_{j=1}^k B_{Ms_{\varepsilon,j}}(x_j) $, then

\begin{equation}\label{nneq4-14}
\begin{split}
|\int_{D_L}\, uV_{\varepsilon,\mathbf{x},i,l}|
\leq  |\int_{B_{s_{\varepsilon,i}}(x_i)}
\,\frac{\partial U_{\varepsilon,x_i,a_{\varepsilon,i}}}{\partial x_{il}} u|
\le  \frac{C s_{\varepsilon,i}}{|\ln s_{\varepsilon,i}|}\|u \|_{L^\infty(D)},
\end{split}
\end{equation}
which gives $|b_{jh  }|\le Cs_{\varepsilon,i}|\ln s_{\varepsilon,i}|^{p-2}\|u \|_{L^\infty(D_L)}$.

\end{remark}

We now consider the following problem

\begin{equation}\label{1-16-1}
Q_{\varepsilon,\mathbf{x}}L_{\varepsilon,\mathbf{x}}(u,\zeta)= f,\quad u\in E_{\varepsilon,\mathbf{x}},
\end{equation}
where $ f\in F_{\varepsilon,\mathbf{x}}$.

We recall
 that if $v\in L^\infty(\mathbb R^2)$ is a solution of
\begin{equation}\label{eq4-5}
-\Delta v -pw_+^{p-1}v=0,\,\,\, \text{in}\,\, \mathbb{R}^2,
\end{equation}
then $\upsilon\in span\{\frac{\partial w}{\partial y_1},\frac{\partial w}{\partial y_2} \}$. Here,
\[
w(y)=\begin{cases}
\phi(|y|),
\,\,\,  & |y|\leq1,\\
\phi^\prime(1)\ln|y|,\,\,\, & |y|>1
\end{cases}
\]
is the solution of
\[-\Delta w=w_+^p,\,\,\, \text{in}\,\, \mathbb{R}^2.\]

\begin{proposition}\label{pro4-2}
Suppose that $(u,\zeta)$ satisfies \eqref{1-16-1}
and \eqref{6-31-8},
with $f=0$ in $D_L\setminus \cup_{j=1}^k B_{Ms_{\varepsilon,j}}(x_j)$ for some large $M>0$. Then
there are constants $C>0$ and $\varepsilon_0>0$, such that for $\varepsilon\in(0,\varepsilon_0)$
\[
\|u\|_{L^\infty(D_L)}+\| u\|_{C^{2,\beta}(D_L\setminus \cup_{j=1}^k B_{\delta}(x_j))}+\|
\zeta\|_{C^{2,\beta}(D_L)}\leq C\bigl( |\ln\ep|^{p-1}\|f\|_{L^\infty(D_L)}+\|f_1\|_{C^{1,\beta}(0,L)}
\bigr).
\]
\end{proposition}

\begin{proof}
We argue by contradiction. Suppose that there are $\varepsilon_n\to 0$, $\mathbf{x}_n\in \prod_{j=1}^k B_\delta(x_{0,j})$, $(u_n,\zeta_n)$  and $(f_n, f_{1n})$, 
such that  \eqref{1-16-1}
and \eqref{6-31-8} holds, $u_n\in E_{\varepsilon_n,\mathbf{x}_n}$,   $f_n:= Q_{\varepsilon_n,\mathbf{x}_n}L_{\varepsilon_n,\mathbf{x}_n}u_n=0$ in $D_L\setminus \cup_{j=1}^k B_{M s_{\varepsilon_n,j}}(x_{n,j})$, 
\begin{equation}\label{eq4-17}
\|f_n
\|_{L^\infty(D_L)}\leq \frac{1}{n|\ln\ep_n|^{p-1}},\;\;\|f_{1n}\|_{C^{1,\beta}(0,L)}\le \frac1n,
\end{equation}
and
\begin{equation}\label{eq4-18}
\|u_n\|_{L^\infty(D_L)}+\|u_n\|_{L^\infty(D_L\setminus \cup_{j=1}^k B_{\delta}(x_j))}
+\|
\zeta_n\|_{C^{2,\beta}(D_L)}=1.
\end{equation}

First, we estimate $b_{n,jh}$ in the following formula:
\begin{equation}\label{eq4-19}
Q_{\varepsilon_n,\mathbf{x}_n}L_{\varepsilon_n,\mathbf{x}_n}(u_n,\zeta_n)=
L_{\varepsilon_n,\mathbf{x}_n}(u_n,\zeta_n)-\sum\limits_{j=1}^k\sum\limits_{h=1}^2
\varepsilon_n^2b_{n,jh}Z_{\varepsilon_n,\mathbf{x}_n,j,h}.
\end{equation}
We know that $b_{n,j1}$ and $b_{n,j2}$ satisfy the following equations

\begin{equation}\label{eq4-20}
\int_{D_L}\,V_{\varepsilon_n,\mathbf{x}_n,i,l}
L_{\varepsilon_n,\mathbf{x}_n}(u_n,\zeta_n)=
\sum\limits_{j=1}^k\sum\limits_{h=1}^2\varepsilon_n^2b_{n,jh}
\int_{D_L}\,Z_{\varepsilon_n,\mathbf{x}_n,j,h}V_{\varepsilon_n,\mathbf{x}_n,i,l}.
\end{equation}
On the other hand, by \eqref{eq3-10} and \eqref{eq3-14}, 

\begin{equation}\label{eq4-22}
\begin{split}
&\int_{D_L}\,V_{\varepsilon_n,\mathbf{x}_n,i,l}
L_{\varepsilon_n,\mathbf{x}_n}(u_n,\zeta_n)
\\
=&\int_{D_L}\,p 
(U_{\varepsilon_n,x_{n,i},a_{\varepsilon_n,i}}-a_{\varepsilon_n,i})_+^{p-1} \bigl(\frac{\partial U_{\varepsilon_n,x_{n,i},a_{\varepsilon_n,i}}}{\partial x_{n,il}}-\frac{\partial a_{\varepsilon_n,i}}{\partial x_{n,il}}\bigr)u_n
\\
&-\sum\limits_{j=1}^k1_{B_\delta(x_{0,j})}f_j'(
\mathcal{U}_{\varepsilon_n,\mathbf{x}_n,\mathbf{a}_n}+\tilde  \psi)\bigl(u_n+\psi_{02,\zeta_n}
\bigr)V_{\varepsilon_n,x_{n,i},l}
\\
&-\int_{D_L}\,\sum\limits_{j=1}^k1_{B_\delta(x_{0,j})}2f_j(
\mathcal{U}_{\varepsilon_n,\mathbf{x}_n,\mathbf{a}_n}+\tilde  \psi)
\frac{\partial \zeta_n}{\partial y_2}V_{\varepsilon_n,x_{n,i},l}
\\
=&o\Bigl( \frac{ s_{\ep_n, i} }{ |\ln s_{\ep_n, i} |^{p-1}}\Bigr).
\end{split}
\end{equation}

Combining \eqref{eq4-20}, \eqref{eq4-22} and \eqref{eq4-13}, we obtain

\[
b_{n,jh}= o\bigl( s_{\ep_n, i} \bigr),
\]
which gives

\begin{equation}\label{nnneq4-19}
\begin{split}
\varepsilon_n^2b_{n,jh}Z_{\varepsilon_n,\mathbf{x}_n,j,h}=&
b_{n,jh}p
(U_{\varepsilon_n,x_{n,i},a_{\varepsilon_n,i}}-a_{\varepsilon_n,i})_+^{p-1} \bigl(\frac{\partial U_{\varepsilon_n,x_{n,i},a_{\varepsilon_n,i}}}{\partial x_{n,il}}-\frac{\partial a_{\varepsilon_n,i}}{\partial x_{n,il}}\bigr)\\
=&O\bigl(\frac{|b_{n,jh}|}{ s_{\ep_n, i} |\ln s_{\ep_n, i}|^{p-1} }\bigr)
=o\bigl(\frac{1}{  |\ln s_{\ep_n, i}|^{p-1} }\bigr).
\end{split}
\end{equation}

From \eqref{eq4-17}, \eqref{eq4-19} and \eqref{nnneq4-19}, we find

\begin{equation}\label{10-17-1}
L_{\varepsilon_n,\mathbf{x}_n}(u_n,\zeta_n)=o\bigl(\frac1{|\ln \ep_n|^{p-1}}\bigr).
\end{equation}

 For any fixed $j$, let $u^*_{n,j}(y)=u_n(x_{n,j}+s_{\varepsilon_n,j}y)$, and let $S_n=\{y:\,x_{n,j}+s_{\varepsilon_n,j}y\in S\}$
 for any set $S$. We define

 \begin{equation}\label{20-17-1}
 \begin{split}
\tilde L_{n}(v,\zeta)=&-\Delta v-\frac{ s^2_{\varepsilon_n,j} }{\ep_n^2}\sum\limits_{j=1}^k  1_{(B_\delta(x_{0,j}))_n} f'_j(\mathcal{U}^*_{\varepsilon_n,\mathbf{x}_n,\mathbf{a}_n}+\tilde\psi^*
)v\\
&-\frac{ s^2_{\varepsilon_n,j} }{\ep_n^2}\sum\limits_{j=1}^k  1_{(B_\delta(x_{0,j}))_n} f'_j(\mathcal{U}^*_{\varepsilon_n,\mathbf{x}_n,\mathbf{a}_n}+\tilde\psi^*)
\psi^*_{02,\zeta_n}\\
&-\frac{ s^2_{\varepsilon_n,j} }{\ep_n^2}\sum\limits_{j=1}^k  1_{(B_\delta(x_{0,j}))_n} 2f_j(\mathcal{U}^*_{\varepsilon_n,\mathbf{x}_n,\mathbf{a}_n}+\tilde\psi^*)
\bigl(\frac{\partial \zeta}{\partial y_2}\bigr)^*.
\end{split}
\end{equation}
We then have

\begin{equation}\label{21-17-1}
\frac{ s^2_{\varepsilon_n,j} }{\ep_n^2}\|\tilde L_{n}(\tilde{u}_{n,j},\tilde \zeta_{n,j})\|_{L^\infty((D_L)_n)}=\| L_{n} (u_{n},\zeta_n)\|_{L^\infty(D_L)}.
\end{equation}
Noting that  $\frac{ s^2_{\varepsilon_n,j} }{\ep_n^2}\sim |\ln \ep_n|^{p-1} $,
we find that

\begin{equation}\label{22-17-1}
\|\tilde L_{n}(\tilde{u}_{n,j},\tilde \zeta_{n,j})\|_{L^\infty((D_L)_n)}=o(1).
\end{equation}
Since $\|\tilde{u}_{n,j}\|_{L^\infty((D_L)_n)}=\|u_n\|_{L^\infty(D_L)}\leq 1$, from \eqref{22-17-1}, we can assume that $\tilde{u}_{n,j}$ converges uniformly in any compact set of $\mathbb{R}^2$ to $u\in L^\infty(\mathbb{R}^2)\cap C(\mathbb{R}^2)$,
and $\zeta_n\to \zeta$ in $C^2(D_L)$. Thus, 

\[
 \psi_{02,\zeta_n}(x_{n,j}+s_{\varepsilon_n,j}y) \to \psi_{02,\zeta}(x_{j}):=c_0.
  \]
  Moreover, from
  \[
  \frac{ s^2_{\varepsilon_n,j} }{\ep_n^2} f_j(\mathcal{U}^*_{\varepsilon_n,\mathbf{x}_n,\mathbf{a}_n}+\tilde\psi^*
)\to 0,
\]
 we see that 
   $u$ satisfies \begin{equation}\label{eq4-28}
-\Delta u-p\phi^{p-1}_+ u-p\phi^{p-1}_+c_0=0, \quad \text{in}\,\, \mathbb{R}^2.
\end{equation}
That is, 

\[
-\Delta (u+c_0)-p\phi^{p-1}_+ (u+c_0)=0, \quad \text{in}\,\, \mathbb{R}^2.
\]
So, we get
\begin{equation}\label{eq4-29}
u+c_0=c_1\frac{\partial w}{\partial y_1}+c_2\frac{\partial w}{\partial y_2}.
\end{equation}
On the other hand, it follows from $u_n\in E_{\varepsilon_n,\mathbf{x}_n}$ that
\begin{equation}\label{eq4-30}
\int_{B_1(0)}\, \phi^{p-1}\frac{\partial\phi}{\partial y_h}u=0,\quad h=1,2.
\end{equation}
We also have

\[
\int_{B_1(0)}\, \phi^{p-1}\frac{\partial\phi}{\partial y_h}c_0=0,\quad h=1,2.
\]
 This gives $c_1=c_2=0$. That is, $u+c_0\equiv0$.  
Thus, we have proved
\begin{equation}\label{eq4-31}
u_n+\psi_{02,\zeta_n}=o(1),\quad \text{in}\,\, \cup_{j=1}^kB_{M's_{\varepsilon_n,j}}(x_{n,j}),
\end{equation}
for any $M'>0$. 
 
By \eqref{eq4-19} and the assumption that
$ Q_{\varepsilon_n,\mathbf{x}_n}L_{\varepsilon_n,\mathbf{x}_n}u_n=0$ in $D_L\setminus \cup_{j=1}^k B_{Ms_{\varepsilon_n,j}}(x_{n,j})$,  we obtain
that

\[
\Delta u_n=0,\quad \text{in}\; D_L\setminus \cup_{j=1}^k B_{Ms_{\varepsilon_n,j}}(x_{n,j}).
\]

Let

\[
\begin{split}
F_n =&\frac1{\ep_n}\Bigl( f_n + \sum_{j=1}^k  1_{B_\delta(x_{0,j})}  f_j'(
\mathcal{U}_{\varepsilon_n,\mathbf{x}_n,\mathbf{a}_n}+\tilde \psi)(u_n+\psi_{02,\zeta_n}
)\\
&+2
f_j(\mathcal{U}_{\varepsilon,\mathbf{x},\mathbf{a}}+\tilde \psi)\frac{\partial \zeta_n}{\partial y_2}
+\sum\limits_{j=1}^k\sum\limits_{h=1}^2
\varepsilon_n^2b_{n,jh}Z_{\varepsilon_n,\mathbf{x}_n,j,h} \Bigr).
\end{split}
\]
Then for $y\in D_L\setminus \cup_{j=1}^k B_{\delta}(x_{n,j})$, it holds

\begin{equation}\label{1-10-5}
\begin{split}
u_n(y)=&\int_{D_L} G_p(z, y) F_n(z)\,dz=\int_{\cup_{j=1}^k B_{Ms_{\varepsilon_n,j}}(x_{n,j})} G_p(z, y) F_n(z)\,dz\\
=& O\Bigl( \int_{\cup_{j=1}^k B_{Ms_{\varepsilon_n,j}}(x_{n,j})} | F_n(z)|\,dz \Bigr).
\end{split}
\end{equation}
But in view of \eqref{eq4-17}, we have

\begin{equation}\label{2-10-5}
\begin{split}
& \int_{ B_{Ms_{\varepsilon_n,j}}(x_{n,j})} |f_n|\le \frac{C s_{\varepsilon_n,j}^2\ep_n^{-2}}{n|\ln\ep_n|^{p-1}}
=o(1).
\end{split}
\end{equation}
We can use \eqref{eq4-31} to derive

\begin{equation}\label{3-10-5}
\begin{split}
& \ep_n^{-2}
 \int_{ B_{Ms_{\varepsilon_n,j}}(x_{n,j})} | f_j'(
\mathcal{U}_{\varepsilon_n,\mathbf{x}_n,\mathbf{a}_n}+\tilde \psi) (u_n+\psi_{02,\zeta_n}
)|\,dz
 \\
 =& o(1)\ep_n^{-2}\int_{ B_{Ms_{\varepsilon_n,j}}(x_{n,j})} | f_j'(
\mathcal{U}_{\varepsilon_n,\mathbf{x}_n,\mathbf{a}_n}+\tilde \psi) |\,dz\\
=&o(1)\ep_n^{-2}\int_{ B_{Ms_{\varepsilon_n,j}}(x_{n,j})}  (U_{\varepsilon_n,x_{n,i},a_{\varepsilon_n,i}}-a_{\varepsilon_n,i})_+^{p-1}\,dz\\
=& o(1)\ep_n^{-2}\int_{ B_{Ms_{\varepsilon_n,j}}(x_{n,j})}  (U_{\varepsilon_n,x_{n,i},a_{\varepsilon_n,i}}-a_{\varepsilon_n,i})_+^{p-1}\,dz\\
=&
o(1) s_{\varepsilon_n,j}^2\ep_n^{-2}  \bigl((\frac{\ep_n}{ s_{\varepsilon_n,j}})^{\frac2{p-1}}\bigr)^{p-1} =o(1).
\end{split}
\end{equation}
It is also easy to verify that

\begin{equation}\label{4-10-5}
\begin{split}
&
|b_{n,jh}| \int_{ B_{Ms_{\varepsilon_n,j}}(x_{n,j})}|
Z_{\varepsilon_n,\mathbf{x}_n,j,h}|\\
=&|b_{n,jh}|O\Bigl( \int_{ B_{Ms_{\varepsilon_n,j}}(x_{n,j})}\ep_n^{-2}
\bigl|  (U_{\varepsilon,x_j,a_{\varepsilon,j}}-a_{\varepsilon,j})_+^{p-1}(\frac{\partial U_{\varepsilon,x_j,a_{\varepsilon,j}}}{\partial x_{jh}}-\frac{\partial a_{\varepsilon,j}}{\partial x_{jh}})\bigr|\Bigr)\\
=& |b_{n,jh}| \ep_n^{-2} s_{\varepsilon_n,j}  \bigl((\frac{\ep_n}{ s_{\varepsilon_n,j}})^{\frac2{p-1}}\bigr)^{p-1}
=o(1).
\end{split}
\end{equation}

Combining \eqref{1-10-5}-\eqref{4-10-5}, we conclude that

\begin{equation}\label{5-10-5}
\begin{split}
u_n(y)=o(1), \quad y\in D_L\setminus \cup_{j=1}^k B_{\delta}(x_{n,j}).
\end{split}
\end{equation}
Moreover, since $u_n$ is harmonic in $D_L\setminus \cup_{j=1}^k B_{Ls_{\varepsilon_n,j}}(x_{n,j})$, we have

\begin{equation}\label{500-10-5}
\|u_n\|_{C^{2,\beta}(D_L\setminus \cup_{j=1}^k B_{2\delta}(x_j))}\le C
\|u_n\|_{L^\infty(D_L\setminus \cup_{j=1}^k B_{\delta}(x_j)))}=o(1).
\end{equation}
Furthermore, 

\begin{equation}\label{501-10-5}
\|\zeta_n\|_{C^{2,\beta}(D_L)}\le C\|u_n\|_{C^{2,\beta}(0,L)}+C\|f_{1n}\|_{C^{1,\beta}(0,L)}=o(1).
\end{equation}

 On the other hand, combining \eqref{eq4-31} and \eqref{501-10-5}, we obtain
 
 \[
u_n=o(1),\quad \text{in}\,\, \cup_{j=1}^kB_{M's_{\varepsilon_n,j}}(x_{n,j}),
\]
which, together with 
  the maximum principle, gives

\[
|u_n(y)|\le \max_{y\in \partial [\cup_{j=1}^k [B_{\delta}(x_{n,j})\setminus   B_{Ls_{\varepsilon_n,j}}(x_{n,j})]]}  |u_n(y)|=o(1),  \quad \forall\; y\in \cup_{j=1}^k [B_{\delta}(x_{n,j})\setminus   B_{Ls_{\varepsilon_n,j}}(x_{n,j}) ].
\]
So we have prove that 

\begin{equation}\label{nneq4-31}
\|u_n\|_{L^\infty(D_L)}=o(1).
\end{equation}
Combining \eqref{nneq4-31}, \eqref{500-10-5} and \eqref{501-10-5}, we obtain a contradiction to \eqref{eq4-18}.

\end{proof}

It follows from Proposition~\ref{pro4-2}   that the solution for \eqref{eq4-1} and \eqref{6-31-8}
is unique.  Next, we prove that \eqref{eq4-1} and \eqref{6-31-8} is solvable. 

\begin{proposition}\label{pro1-10-5}
Problem \eqref{eq4-1}-\eqref{6-31-8} is uniquely solvable.
\end{proposition}

\begin{proof}

First, given any $\omega \in E_{\varepsilon,\mathbf{x} }  $, 
\eqref{6-31-8} has a unique solution $\zeta=\zeta(\omega, f_1)$.
Then, we insert $\zeta=\zeta(\omega, f_1)$ into \eqref{eq4-1}.
Note that if $f\in F_{\varepsilon,\mathbf{x} } $,  solution $u$ of $-\Delta u= f$ satisfies $u\in E_{\varepsilon,\mathbf{x} }  $.
Now, we rewrite \eqref{1-16-1} as

\[
\begin{split}
 \omega-\ep^{-2}(-\Delta )^{-1} Q_{\varepsilon,\mathbf{x}}\bigl[\sum\limits_{j=1}^k  1_{B_\delta(x_{0,j})}  \Bigl[&f'_j(\mathcal{U}_{\varepsilon,\mathbf{x},\mathbf{a}}+\tilde \psi)(\omega+\psi_{02,\zeta}
)\\
&+2
f_j(\mathcal{U}_{\varepsilon,\mathbf{x},\mathbf{a}}+\tilde \psi)\frac{\partial \zeta}{\partial y_2}
\Bigr]\bigr]=\ep^{-2}(-\Delta)^{-1} f.
 \end{split}
\]

Let

\[
\begin{split}
\hat T \omega=(-\Delta )^{-1} Q_{\varepsilon,\mathbf{x}}\bigl[\sum\limits_{j=1}^k  1_{B_\delta(x_{0,j})} \Bigl[&f'_j(\mathcal{U}_{\varepsilon,\mathbf{x},\mathbf{a}}+\tilde \psi)(\omega+\psi_{02,\zeta}
)\\
&+2
f_j(\mathcal{U}_{\varepsilon,\mathbf{x},\mathbf{a}}+\tilde \psi)\frac{\partial \zeta}{\partial y_2}
\Bigr]\bigr],\quad \omega\in
E_{\varepsilon,\mathbf{x} }.
\end{split}
\]

We claim that $\hat T$ is compact in $E_{\varepsilon,\mathbf{x} }$.
In fact, suppose that $\omega_n\in F_{\varepsilon,\mathbf{x} }$ and $\|\omega_n\|_{C^{2,\beta}(D_L)}\le C$. Denote  $h_n = \hat T\omega_n\in E_{\varepsilon,\mathbf{x} }$. Then

\begin{equation}\label{1-2-9}
\begin{split}
-\Delta h_n= \ep^{-2}\sum\limits_{j=1}^k  1_{B_\delta(x_{0,j})} \Bigl[&f'_j(\mathcal{U}_{\varepsilon,\mathbf{x},\mathbf{a}}+\tilde \psi)(\omega_n+\psi_{02,\zeta_n}
)\\
&+2
f_j(\mathcal{U}_{\varepsilon,\mathbf{x},\mathbf{a}}+\tilde \psi)\frac{\partial \zeta_n}{\partial y_2}
\Bigr]
\end{split}
\end{equation}
Since $\|\omega_n\|_{C^{2,\beta}(D_L)}\le C$, we obtain that
$\|\zeta_n\|_{C^{2,\beta}(D_L)}\le C$. This shows that the function in the right hand side of
\eqref{1-2-9} is bounded in $C^{1,\beta}(D_L)$, which in turns implies that
$h_n$ is bounded in $C^{2,\beta_1}(D_L)$ for any $\beta_1$ close to 1.
As a result,  $h_n$ is compact in $C^{2,\beta}(D_L)$.

It follows from Proposition~\ref{pro4-2} that $I-\hat T$ is an injective
map. So Fredholm theorem gives that $I-\hat T$ is bijective.

\end{proof}

By Proposition~\ref{pro1-10-5}, given any $(f,f_1)$, we denote

\begin{equation}\label{10-2-9}
 T(f,f_1)= ( T_1(f,f_1), T_2(f,f_1))
\end{equation}
the solution of \eqref{eq4-1}-\eqref{6-31-8}.

\subsection{Case $\alpha>0$}

Let

\begin{equation}\label{0-4-9}
\begin{split}
\hat L_{\varepsilon,\mathbf{x}}(\omega,\zeta):=-\ep^2\Delta v-\sum\limits_{j=1}^k  1_{B_\delta(x_{0,j})} \Bigl[&f'_j(\mathcal{U}_{\varepsilon,\mathbf{x},\mathbf{a}}+\tilde \psi)\omega\\
&+2
f_j(\mathcal{U}_{\varepsilon,\mathbf{x},\mathbf{a}}+\tilde \psi)\frac{\partial \zeta}{\partial y_2}
\Bigr],\quad \omega\in E_{\varepsilon,\mathbf{x}}
\end{split}
\end{equation}
In \eqref{10-21-8},  taking $u=\mathcal{U}_{\varepsilon,\mathbf{x},\mathbf{a}}
+\omega$,  we find

\[
  \bigl(\frac{\partial u}{\partial y_2}\bigr)^2= \bigl(\frac{\partial \mathcal{U}_{\varepsilon,\mathbf{x},\mathbf{a}}}{\partial y_2}\bigr)^2+
  2\frac{\partial \mathcal{U}_{\varepsilon,\mathbf{x},\mathbf{a}}}{\partial y_2}\frac{\partial \omega}{\partial y_2}+\bigl(\frac{\partial \omega}{\partial y_2}\bigr)^2,
  \]
in which, the linear term for $\omega$  can be ignored since 
$\frac{\partial \mathcal{U}_{\varepsilon,\mathbf{x},\mathbf{a}}}{\partial y_2}=O\bigl(
\frac1{|\ln\ep|}\bigr)$ on $y_2=\pi$.  Thus, the linear problem for $\alpha>0$ is 

\begin{equation}\label{1-4-9}
\begin{split}
Q_{\varepsilon,\mathbf{x}} \hat L_{\varepsilon,\mathbf{x}}(\omega,\zeta)=f,\quad \omega\in E_{\varepsilon,\mathbf{x}}
\end{split}
\end{equation}
and

\begin{equation}\label{2-4-9}
\begin{split}
-\alpha^2 \zeta'' +g\zeta= f_1,\quad \text{on}\; y_2=\pi.
\end{split}
\end{equation}
Here, in \eqref{1-4-9}, $\zeta$ is the  harmonic extension
of the solution of \eqref{2-4-9},  in $D_L$, satisfying $\zeta=0$ on $y_2=0$.

The linear problem \eqref{1-4-9}-\eqref{2-4-9} is simpler than that for
the case $\alpha=0$.  Similar to the last subsection, we can prove the following result.

\begin{proposition}\label{p1-4-9}
Problem \eqref{1-4-9}-\eqref{2-4-9} is uniquely solvable. Moreover,

\[
\|\omega\|_{L^\infty(D_L)}+\| \omega\|_{C^{2,\beta}(D_L\setminus \cup_{j=1}^k B_{\delta}(x_j))}+\|
\zeta\|_{C^{2,\beta}(D_L)}\leq C\bigl( |\ln\ep|^{p-1}\|f\|_{L^\infty(D_L)}+\|f_1\|_{C^{\beta}(0,L)}
\bigr).
\]
\end{proposition}

By Proposition~\ref{p1-4-9}, given any $(f,f_1)$, we denote

\begin{equation}\label{11-4-9}
\hat  T(f,f_1)= ( \hat T_1(f,f_1), \hat T_2(f,f_1))
\end{equation}
the solution of \eqref{1-4-9}-\eqref{2-4-9}.

\section{The Reduction}

\subsection{ The case $\alpha=0$}
We will look for solution for \eqref{n10-24-8} and \eqref{21-24-8}  of the form
\[
u_\varepsilon=\mathcal{U}_{\varepsilon,\mathbf{x},\mathbf{a}}
+\omega_{\varepsilon,\mathbf{x}},\quad \Gamma_2=\Gamma_0+\zeta
\]
where 
 $\Gamma_0$ is the solution of
\eqref{8-29-8},  and $\omega_{\varepsilon,\mathbf{x}}\in E_{\varepsilon,\mathbf{x}}$.
We write \eqref{n10-24-8} as

  \begin{equation}\label{10-11-5}
L_{\varepsilon,\mathbf{x}}(\omega,\zeta)=l_{\varepsilon,\mathbf{x}}
+R_{\varepsilon,\mathbf{x}}(\omega, \zeta),
\end{equation}
where $L_{\varepsilon,\mathbf{x}}$ is defined in \eqref{eq4-1},

\begin{equation}\label{eq4-3}
\begin{split}
l_{\varepsilon,\mathbf{x}}=
\sum\limits_{j=1}^k1_{B_\delta(x_{0,j})}
f_j\big(\mathcal{U}_{\varepsilon,\mathbf{x},\mathbf{a}}+\tilde \psi\big)
+\ep^2\Delta U_{\varepsilon,x_j,a_{\varepsilon,j}}+ 2\frac{\partial \Gamma_0}{\partial y_2}\sum\limits_{j=1}^k1_{B_\delta(x_{0,j})}
f_j\big(\mathcal{U}_{\varepsilon,\mathbf{x},\mathbf{a}}+ \tilde \psi
\big)
\end{split}
\end{equation}

\begin{equation}\label{eq4-4}
\begin{split}
R_{\varepsilon,\mathbf{x}}(\omega, \zeta)=
&\sum\limits_{j=1}^k1_{B_\delta(x_{0,j})}\Big(
f_j\big(\mathcal{U}_{\varepsilon,\mathbf{x},\mathbf{a}}+ \psi_0
+\omega\big)
-f_j\big(\mathcal{U}_{\varepsilon,\mathbf{x},\mathbf{a}}+\tilde \psi\big)
\\
&-f'_j\big(\mathcal{U}_{\varepsilon,\mathbf{x},\mathbf{a}}+\tilde \psi
\big)(\omega+\psi_{02,\zeta}
)\Big)\\
&+\bigl( 2\frac{\partial \zeta}{\partial y_2}+\bigl(\frac{\partial \Gamma_2}{\partial y_2}\bigr)^2 
+
\bigl(\frac{\partial \Gamma_2}{\partial y_1}\bigr)^2
\bigr)\sum\limits_{j=1}^k1_{B_\delta(x_{0,j})}
f_j\big(\mathcal{U}_{\varepsilon,\mathbf{x},\mathbf{a}}+ \psi_0
+\omega\big),
\end{split}
\end{equation}

 For the  free boundary equation \eqref{21-24-8}, we choose $\gamma>0$, such that
$\frac{4(g+\frac\pi2 \gamma^2 L)  }{\pi^2\gamma^2  }\ne \frac{n\pi}{L} \coth\frac{n\pi^2}{L}$, $n\ge 1$.
On $ \{y_2=\pi\}$, $\zeta$ satisfies

\begin{equation}\label{10-28-8}
\begin{split}
&-\frac{\pi^2\gamma^2}4 \frac{\partial \zeta}{\partial y_2}
+\frac{\pi^2\gamma^2}{4L} \int_{0}^L\frac{\partial \zeta}{\partial y_2}+\bigl(g+\frac\pi2 \gamma^2 L
  \bigr)\zeta-\bigl[\frac{\pi \gamma}2 \frac{\partial \omega}{\partial y_2}
  -\frac{\pi \gamma}2 \frac1L \int_{0}^L\frac{\partial \omega}{\partial y_2}\bigr]\\
  = &
\mathcal F(\mathcal{U}_{\varepsilon,\mathbf{x},\mathbf{a}}
+\omega,\Gamma_0+\zeta)-  \frac1L \int_{0}^L\mathcal F(\mathcal{U}_{\varepsilon,\mathbf{x},\mathbf{a}}
+\omega,\Gamma_0+\zeta),
  \end{split}
\end{equation}
where 

\begin{equation}\label{1-3-9}
\mathcal F(u,\Gamma)=  \frac12 \frac{1}{|1+\Gamma'|^2} \bigl(\frac{\partial (u+\psi_0)}{\partial y_2}\bigr)^2  -\Bigl( -\frac{\pi^2\gamma^2}4 \frac{\partial \Gamma_2}{\partial y_2}+\bigl(g+\frac\pi2 \gamma^2 L
  \bigr)\Gamma_2  \Bigr),
\end{equation}
which satisfies   \eqref{0-24-8}. See \eqref{62-24-8}.

We consider the following problem

\begin{equation}\label{1-31-8}
\begin{cases}
\Delta \zeta=0,\;\; \text{in}\; D_L,\\
\zeta=0, \;\; \text{on}\; \{y_2=0\} \\
-\frac{\pi^2\gamma^2}4 \frac{\partial \zeta}{\partial y_2}
+\frac{\pi^2\gamma^2}{4L} \int_{0}^L\frac{\partial \zeta}{\partial y_2}+\bigl(g+\frac\pi2 \gamma^2 L
  \bigr)\zeta
  -\bigl[\frac{\pi \gamma}2 \frac{\partial \omega}{\partial y_2}
  -\frac{\pi \gamma}2 \frac1L \int_{0}^L\frac{\partial \omega}{\partial y_2}\bigr]= \text{RHS of \eqref{10-28-8}}, \text{on}\; \{y_2=\pi\},\\
\text{ $\zeta$  is periodic in $y_1$ in $D_L$}.
\end{cases}
\end{equation}

We denote

\begin{equation}\label{0-3-9}
\begin{split}
&\mathbf T(\omega,\zeta)= (\mathbf T_1(\omega,\zeta), \mathbf T_2(\omega,\zeta))\\
\\
=&T\Bigl( Q_{\varepsilon,\mathbf{x}}
\bigl(l_{\varepsilon,\mathbf{x}}
+R_{\varepsilon,\mathbf{x}}(\omega,\zeta)\bigr),\; \mathcal F(\mathcal{U}_{\varepsilon,\mathbf{x},\mathbf{a}}
+\omega,\Gamma_0+\zeta)-  \frac1L \int_{0}^L\mathcal F(\mathcal{U}_{\varepsilon,\mathbf{x},\mathbf{a}}
+\omega,\Gamma_0+\zeta)\Bigr),
\end{split}
\end{equation}
where $T$ is defined in \eqref{10-2-9}.

We consider the following problem

\begin{equation}\label{1-24-1}
(\omega,\zeta)=\mathbf T(\omega,\zeta).
\end{equation}

\bigskip

\begin{lemma}\label{l2-12-5}
We have

\[
\|l_{\varepsilon,\mathbf{x}}\|_{L^\infty(D_L)}\le \frac{C }{|\ln \varepsilon|^{p+1}}
\]

\end{lemma}

\begin{proof}
 Let $y\in  B_{ M s_{\varepsilon,j}  }(x_{0,j})$.  If $\tau_j=0$, then using \eqref{eq3-10},  we can prove that

\begin{equation}\label{eq4-42a}
\begin{split}
&\big(\mathcal{U}_{\varepsilon,\mathbf{x},\mathbf{a}}+\tilde \psi
-\kappa_j\big)_+^{p}-(U_{\varepsilon,x_j,a_{\varepsilon,j}}
-a_{\varepsilon,j})_+^p
\\
=&(U_{\varepsilon,x_j,a_{\varepsilon,j}}
-a_{\varepsilon,j}+O(\frac{s_{\varepsilon,j}}{|\ln \varepsilon|}))_+^p-(U_{\varepsilon,x_j,a_{\varepsilon,j}}-a_{\varepsilon,j})_+^p
\\
=&O\Big((U_{\varepsilon,x_j,a_{\varepsilon,j}}
-a_{\varepsilon,j})_+^{p-1}\frac{s_{\varepsilon,j}}{|\ln \varepsilon|}\Big)
\\
=&O\Big(\frac{s_{\varepsilon,j}}{|\ln \varepsilon|^p}\Big),
\end{split}
\end{equation}
while if $\tau_j=1$,

\begin{equation}\label{eq4-42a}
\begin{split}
&-\big(-(\mathcal{U}_{\varepsilon,\mathbf{x},\mathbf{a}}+\tilde \psi)
-\kappa_j\big)_+^{p}+(U_{\varepsilon,x_j,a_{\varepsilon,j}}
-a_{\varepsilon,j})_+^p
\\
=&-(U_{\varepsilon,x_j,a_{\varepsilon,j}}-a_{\varepsilon,j}+O(\frac{s_{\varepsilon,j}}{|\ln \varepsilon|}))_+^p+(U_{\varepsilon,x_j,a_{\varepsilon,j}}-a_{\varepsilon,j})_+^p
\\
=&O\Big((U_{\varepsilon,x_j,a_{\varepsilon,j}}-a_{\varepsilon,j})_+^{p-1}\frac{s_{\varepsilon,j}}{|\ln \varepsilon|}\Big)
\\
=&O\Big(\frac{s_{\varepsilon,j}}{|\ln \varepsilon|^p}\Big).
\end{split}
\end{equation}

We also have

\[
\frac{\partial \Gamma_0}{\partial y_2}\sum\limits_{j=1}^k1_{B_\delta(x_{0,j})}
f_j\big(\mathcal{U}_{\varepsilon,\mathbf{x},\mathbf{a}}+ \tilde \psi
\big)=O\Big(\frac{1}{|\ln \varepsilon|^{p+1}}\Big),
\]
and thus the result follows.

\end{proof}

To estimate $R_{\varepsilon,\mathbf{x}}$, we need to estimate
$\psi_0-\tilde \psi$. 

\begin{lemma}\label{l1-29-8}
We  have

\[
\psi_0=\tilde \psi +\psi_{02,\zeta}+\psi_{03}, 
\]
where $\psi_{03}$ satisfies

\[
\|\psi_{03}\|_{L^\infty(D_L)}=O\Bigl( \frac{1}{|\ln s_\ep|^2} \Bigr).
\]
\end{lemma}

\begin{proof}
Let  $\psi=\psi_0-\tilde \psi$.  From \eqref{20-2-9} and \eqref{100-29-8}, we have

\[
-\Delta\psi= 2 \frac{\partial \zeta}{\partial y_2} \gamma +\Bigl( \bigl(\frac{\partial \Gamma_2}{\partial y_1}\bigr)^2  +\bigl(\frac{\partial \Gamma_2}{\partial y_2}\bigr)^2 \Bigr)\gamma
\]
Note that we have $|\Gamma_2|\le \frac{C}{|\ln s_\ep|}$. 
Thus, the result follows.
\end{proof}
\begin{lemma}\label{l3-12-5}
 We have

\begin{equation}\label{31-28-8}
\begin{split}
\|R_{\varepsilon,\mathbf{x}}(\omega,\zeta)\|_{L^\infty(D_L)}
\le &  \frac C{|\ln s_\ep|^{p+1} }+\frac{C\|\omega\|_{L^\infty(\Omega)}^{\min \{2,p\}}
}{ |\ln\varepsilon|^{\max\{0,p-2\}}}\\
&+C\Bigl( \frac1{|\ln \varepsilon|^{p}}
+\|\omega\|_{L^\infty(D_L)}^p 
\Bigr)\|\zeta\|_{C^1(D_L)}.
\end{split}
\end{equation}

\end{lemma}

\begin{proof}

By Lemma~\ref{l3-12-5}, we have

\begin{equation}\label{eq4-43a}
\begin{split}
&
f_j\big(\mathcal{U}_{\varepsilon,\mathbf{x},\mathbf{a}}+ \psi_0
+\omega\big)
-f_j\big(\mathcal{U}_{\varepsilon,\mathbf{x},\mathbf{a}}+\tilde \psi\big)
-f'_j\big(\mathcal{U}_{\varepsilon,\mathbf{x},\mathbf{a}}+\tilde \psi
\big)(\omega+\psi_{02}(\zeta)
)\\
=&O\Big(
\big(\mathcal{U}_{\varepsilon,\mathbf{x},\mathbf{a}}+\tilde \psi-\kappa_j\big)_+^{\max \{0,p-2\}}\|\omega\|_{L^\infty(\Omega)}^{\min \{2,p\}}+
O\Big(
\big(\mathcal{U}_{\varepsilon,\mathbf{x},\mathbf{a}}+\tilde \psi-\kappa_j\big)_+^{ p-1}\|\psi_{03}\|_{L^\infty(\Omega)}
\\
=&O\Big(\frac{\|\omega\|_{L^\infty(\Omega)}^{\min \{2,p\}}}
{ |\ln\varepsilon|^{\max\{0,p-2\} }}\Big)+O\Big(\frac{1}
{ |\ln\varepsilon|^{p+1 }}\Big),
\end{split}
\end{equation}
if $\tau_j=0$.  while for $\tau_j=1$, we can prove in a similar way that 
\eqref{eq4-43a} also holds.

 On the other hand,  we have
\begin{equation}\label{1-13-5}
\begin{split}
&\bigl( 2\frac{\partial \zeta}{\partial y_2}+\bigl(\frac{\partial \Gamma_2}{\partial y_2}\bigr)^2 +
\bigl(\frac{\partial \Gamma_2}{\partial y_1}\bigr)^2
\bigr)
f_j\big(\mathcal{U}_{\varepsilon,\mathbf{x},\mathbf{a}}+\tilde \psi
+\omega\big)\\
=& O\Bigl( \frac1{|\ln \varepsilon|^{p}}
+\|\omega\|_{L^\infty(D_L)}^p 
\Bigr)\bigl(\|\zeta\|_{C^1(D_L)}+\|\Gamma_2\|^2_{C^1(D_L)}\bigr).
\end{split}
\end{equation}

Combining  \eqref{eq4-43a} and \eqref{1-13-5}, we obtain \eqref{31-28-8}.

\end{proof}

\begin{lemma}\label{l1-28-8}
We have

\[
\|\mathcal F(u,\Gamma)\|_{C^{1,\beta}(0,L)}\le
\frac{C}{|\ln \ep|^2}+C\|\omega\|^2_{C^{2,\beta}(0,L)}+C\|\zeta \|^2_{C^{2,\beta}(0,L)},
\]
where $\mathcal F(u,\Gamma)$ is given in \eqref{1-3-9}.
\end{lemma}

\begin{proof}
 We have

\[
\begin{split}
|\mathcal F(\mathcal{U}_{\varepsilon,\mathbf{x},\mathbf{a}}
+\omega,\Gamma)|\le & C(\frac{\partial u}{\partial y_2})^2+C\Gamma_2^2+C
|\nabla \Gamma_2|^2
\\
\le & \frac{C}{|\ln\ep|^2}+C|\nabla\omega|^2+C|\zeta|^2+C|\nabla\zeta|^2.
\end{split}
\]
This gives

\[
\begin{split}
&\|\mathcal F(\mathcal{U}_{\varepsilon,\mathbf{x},\mathbf{a}}
+\omega,\Gamma)\|_{C(0,L)}\le
\frac{C}{|\ln \ep|^2}+C\| \omega\|^2_{C^{1}(0,L)}+C\|\zeta \|^2_{C^{1}(0,L)}.
\end{split}
\]

Using \eqref{1-3-9}, we can also show that

\[
\begin{split}
&\|\mathcal F(\mathcal{U}_{\varepsilon,\mathbf{x},\mathbf{a}}
+\omega,\Gamma)\|_{C^{1,\beta}(0,L)}\le
\frac{C}{|\ln \ep|^2}+C\|\omega\|^2_{C^{2,\beta}(0,L)}+C\|\zeta \|^2_{C^{2,\beta}(0,L)}.
\end{split}
\]

\end{proof}

Let
\[
M_1=\{\omega:\, \omega\in E_{\varepsilon,\mathbf{x}},\,\,
\|\omega\|_{L^\infty(D_L)}+\|\omega\|_{C^{2,\beta}(D_L\setminus \cup_{j=1}^k B_{\delta}(x_j))}
\leq \frac1{|
\ln \varepsilon|^{1+\sigma}} \},
\]
and

\[
M_2=\{\zeta:  \zeta\in C^{2,\beta}(0,L),\; \|\zeta\|_{C^{2,\alpha}(0,L)}\le  \frac1{|
\ln \varepsilon|^{1+\sigma}}  \},
\]
where $\sigma>0$ is a small constant.  We are ready to prove that
the operator $\mathbf T$ defined in \eqref{0-3-9}  has a unique fixed point
in $ M_1\times M_2$.

\begin{proposition}\label{pro4-4}
There is an $\varepsilon_0>0$ such that for any $\varepsilon\in (0,\varepsilon_0)$ and  $\mathbf{x} \in \prod_{j=1}^k B_\delta(x_{0,j})$, there exist uniquely
$\omega_{\varepsilon,\mathbf{x}}\in E_{\varepsilon,\mathbf{x}}$
and $\zeta_{\varepsilon,\mathbf{x}}$, satisfying
\[
(\omega_{\varepsilon,\mathbf{x}}, \zeta_{\varepsilon,\mathbf{x}})=\mathbf T
(\omega_{\varepsilon,\mathbf{x}}, \zeta_{\varepsilon,\mathbf{x}}).
\]
Moreover, 
\begin{equation}\label{26-3-9}
\|\omega_{\varepsilon,\mathbf{x}}\|_{L^\infty(\Omega)}+
\| \omega_{\varepsilon,\mathbf{x}}\|_{C^{2,\beta}(D_L\setminus \cup_{j=1}^k B_{\delta}(x_j))}
+\|\zeta_{\varepsilon,\mathbf{x}}\|_{C^{2,\beta}(0,L)   }\leq  \frac{C}{|\ln\varepsilon|^{1+\sigma}}.
\end{equation}

\end{proposition}

\begin{proof}
We will use the contraction mapping theorem to prove this proposition.

Firstly, we show that $\mathbf T$ maps $M_1\times M_2$ to itself.

For any $(\omega, \phi)\in M_1\times M_2$,  we have that
\begin{equation}\label{eq4-41}
l_{\varepsilon,\mathbf{x}}=R_{\varepsilon,\mathbf{x}}(\omega,\zeta)=0,\quad \text{in}\,\, \Omega\setminus \cup_{j=1}^k B_{Ls_{\varepsilon,j}(x_j)}.
\end{equation}
Then we have
\[
Q_{\varepsilon,\mathbf{x}}(l_{\varepsilon,\mathbf{x}}
+R_{\varepsilon,\mathbf{x}}(\omega,\zeta))=0,\quad \text{in}\,\, \Omega\setminus \cup_{j=1}^k B_{Ls_{\varepsilon,j}(x_j)}.
\]
Therefore, by  Proposition~\ref{pro4-2} we get
\begin{equation}\label{eq4-42}
\begin{split}
&\|\mathbf T_1(\omega,\zeta)\|_{L^\infty(D_L)}+\| \mathbf T_1(\omega,\zeta)\|_{C^{2,\beta}(D_L\setminus \cup_{j=1}^k B_{\delta}(x_j))}+\|
\mathbf T_2(\omega,\zeta)\|_{C^{2,\beta}(D_L)}
\\
\leq  &C |\ln \ep|^{p-1} \Bigl(\|l_{\varepsilon,\mathbf{x}}\|_{L^\infty}
+\|R_{\varepsilon,\mathbf{x}}(\omega,\zeta)\|_{L^\infty}\Bigr)+C
\|\mathcal F(u,\Gamma)\|_{C^{1,\beta}(0,L)}.
\end{split}
\end{equation}

On the other hand, by Lemma~\ref{l3-12-5},

\begin{equation}\label{10-3-9}
\begin{split}
\|R_{\varepsilon,\mathbf{x}}(\omega,\zeta)\|_{L^\infty(D_L)}
\le &  \frac C{|\ln s_\ep|^{p+1} }+\frac{C\|\omega\|_{L^\infty(\Omega)}^{\min \{2,p\}}
}{ |\ln\varepsilon|^{\max\{0,p-2\}}}
\\&+C\Bigl( \frac1{|\ln \varepsilon|^{p}}
+\|\omega\|_{L^\infty(D_L)}^p 
\Bigr)\|\zeta\|_{C^1(D_L)}\\
\le & \frac C{|\ln \ep|^{p+1} }+ \frac{C
}{ |\ln\varepsilon|^{\max\{0,p-2\}}}
\frac{1}{  |\ln \ep|^{(1+\sigma)(\min \{2,p\})}}\\
\le & \frac{C}{|\ln\ep|^{p+\min(p,2)\sigma}}.
\end{split}
\end{equation}
which, together with Lemmas~\ref{l2-12-5}  and \ref{l1-28-8},
gives

\begin{equation}\label{eq4-45}
\begin{split}
\|\mathbf T_1(\omega,\zeta)\|_{L^\infty(D_L)}+\| \mathbf T_1(\omega,\zeta)\|_{C^{2,\beta}(D_L\setminus \cup_{j=1}^k B_{\delta}(x_j))}+\|
\mathbf T_2(\omega,\zeta)\|_{C^{2,\beta}(D_L)}
\le  \frac{1}{  |\ln \ep|^{1+\sigma}}.
\end{split}
\end{equation}
if $\sigma>0$ is small.
Thus, $T$ maps $M_1\times M_2$ to itself.

\medskip

Next, we show that $\mathbf T$ is a contraction map.
For any $(\omega_1,\zeta_1),  (\omega_2,\zeta_2)\in M_1\times M_2$, we have

\[
\begin{split}
&\mathbf T(\omega_1,\zeta_1)-\mathbf T(\omega_2,\zeta_2)\\
=&T\Bigl(Q_{\varepsilon,\mathbf{x}}
R_{\varepsilon,\mathbf{x}}(\omega_1,\zeta_1),\mathcal F_1(\omega_1,\zeta_1)\Bigr)-T\Bigl(Q_{\varepsilon,\mathbf{x}}
R_{\varepsilon,\mathbf{x}}(\omega_2,\zeta_2),\mathcal F_1(\omega_2,\zeta_2)\Bigr)\\
%=&T\Bigl(Q_{\varepsilon,\mathbf{x}}
%R_{\varepsilon,\mathbf{x}}(\omega_1,\zeta_1),\mathcal F_1(\omega_1,\zeta_1)\Bigr)-T\Bigl(Q_{\varepsilon,\mathbf{x}}
%R_{\varepsilon,\mathbf{x}}(\omega_2,\zeta_2),\mathcal F_1(\omega_2,\zeta_2)\Bigr)\\
=&T\Bigl( Q_{\varepsilon,\mathbf{x}}
R_{\varepsilon,\mathbf{x}}(\omega_1,\zeta_1),\,
 \mathcal F_1(\omega_1,\zeta_1)- \mathcal F_1(\omega_2,\zeta_2)\Bigr)\\
 &+T\Bigl( Q_{\varepsilon,\mathbf{x}}
R_{\varepsilon,\mathbf{x}}(\omega_1,\zeta_1)- Q_{\varepsilon,\mathbf{x}}
R_{\varepsilon,\mathbf{x}}(\omega_2,\zeta_2),\,
 \mathcal F_1(\omega_2,\zeta_2)\Bigr),
\end{split}
\]
where

\[
\mathcal F_1(\omega,\zeta)=\mathcal F(\mathcal{U}_{\varepsilon,\mathbf{x},\mathbf{a}}
+\omega,\Gamma_0+\zeta)-  \frac1L \int_{0}^L\mathcal F(\mathcal{U}_{\varepsilon,\mathbf{x},\mathbf{a}}
+\omega,\Gamma_0+\zeta).
\]

 From  Proposition~\ref{pro4-2}, we obtain
\begin{equation}\label{eq4-46}
\begin{split}
&\| \mathbf T_1(\omega_1,\zeta_1)-\mathbf T_1(\omega_2,\zeta_2)\|_{L^\infty(D_L)}+\| \mathbf T_1(\omega_1,\zeta_1)-\mathbf T_1(\omega_2,\zeta_2)\|_{C^{2,\beta}(D_L\setminus \cup_{j=1}^k B_{\delta}(x_j))}\\
&+\|
\mathbf T_2(\omega_1,\zeta_1)-\mathbf T_2(\omega_2,\zeta_2)\|_{C^{2,\beta}(D_L)}\\
\leq  &C |\ln \ep|^{p-1} 
\|R_{\varepsilon,\mathbf{x}}(\omega_1,\zeta_1)-R_{\varepsilon,\mathbf{x}}(\omega_2,\zeta_2)
\|_{L^\infty(D_L)}\\
&+C
\|\mathcal F_1(\omega_1,\zeta_1)-\mathcal F_1(\omega_2,\zeta_2)\|_{C^{1,\beta}(0,L)}.
\end{split}
\end{equation}
To estimate $
\|R_{\varepsilon,\mathbf{x}}(\omega_1,\phi_1)-
R_{\varepsilon,\mathbf{x}}(\omega_2,\phi_2)\|_{L^\infty(D_L)}
$, we proceed as follows.

Note that the background flow $\psi_0$ defined in \eqref{2-15-6} depends on $\Gamma_0+\zeta$.
We denote by $\psi_{0,\zeta}$ the background flow.
We have

\[
\psi_{0,\zeta}=\tilde \psi +\psi_{02,\zeta} +\psi_{03,\zeta},
\]
where $\psi_{03,\zeta}$ is the solution of 

\[
-\Delta \psi_{03,\zeta}= \Bigl( \bigl(\frac{\partial (\Gamma_0+\zeta)}{\partial y_1}\bigr)^2  +\bigl(\frac{\partial (\Gamma_0+\zeta)}{\partial y_2}\bigr)^2 \Bigr)\gamma.
\]
Thus, if $\tau_j=0$,
\begin{equation}\label{neq4-43a}
\begin{split}
&
\Bigl[f_j\big(\mathcal{U}_{\varepsilon,\mathbf{x},\mathbf{a}}+ \psi_{0,\zeta_1}
+\omega_1\big)
-f_j\big(\mathcal{U}_{\varepsilon,\mathbf{x},\mathbf{a}}+\tilde \psi\big)
-f'_j\big(\mathcal{U}_{\varepsilon,\mathbf{x},\mathbf{a}}+\tilde \psi
\big)(\omega_1+\psi_{02,\zeta_1})\Bigr]\\
&-\Bigl[f_j\big(\mathcal{U}_{\varepsilon,\mathbf{x},\mathbf{a}}+ \psi_{0,\zeta_2}
+\omega_2\big)
-f_j\big(\mathcal{U}_{\varepsilon,\mathbf{x},\mathbf{a}}+\tilde \psi\big)
-f'_j\big(\mathcal{U}_{\varepsilon,\mathbf{x},\mathbf{a}}+\tilde \psi
\big)(\omega_2+\psi_{02,\zeta_2})\Bigr]
\\
=&\big(\mathcal{U}_{\varepsilon,\mathbf{x},\mathbf{a}}+ \psi_{0,\zeta_1}
+\omega_1-\kappa_j
\big)_+^{p}-\big(\mathcal{U}_{\varepsilon,\mathbf{x},\mathbf{a}}+ \psi_{0,\zeta_2}
+\omega_2-\kappa_j
\big)_+^{p}\\
&
-p\big(\mathcal{U}_{\varepsilon,\mathbf{x},\mathbf{a}}+\tilde \psi-\kappa_j\big)_+^{p-1}[(\omega_1-
\omega_2)+ (\psi_{02,\zeta_1}  - \psi_{02,\zeta_2})]
\\
=&p\big(\mathcal{U}_{\varepsilon,\mathbf{x},\mathbf{a}}+\tilde \psi-\kappa_j\big)_+^{p-1}[\psi_{03,
\zeta_1}-\psi_{03,
\zeta_2}]
\\
&+O\Big(
\big(\mathcal{U}_{\varepsilon,\mathbf{x},\mathbf{a}}+\tilde \psi-\kappa_j\big)_+^{\max \{0,p-2\}}\bigl(\|\omega_1\|_{L^\infty(D_L)}^{\min \{1,p-1\}}+\|\omega_2\|_{L^\infty(D_L)}^{\min \{1,p-1\}}\\
&\qquad\qquad+\|\zeta_1\|_{L^\infty(D_L)}^{\min \{1,p-1\}}+\|\zeta_2\|_{L^\infty(D_L)}^{\min \{1,p-1\}}\bigr)\Big)\\
&\times\Bigl[\|\omega_1-\omega_2\|_{L^\infty(D_L)}
 +\|\psi_{02,\zeta_1}  - \psi_{02,\zeta_2}\|_{L^\infty(D_L)}
+\|\psi_{03,\zeta_1}  - \psi_{03,\zeta_2}\|_{L^\infty(D_L)}\Bigr]
\\
=&O\Big(\frac{1
}{ |\ln\varepsilon|^{p-1+\sigma_1}}\Big)\Big(\|\omega_1-\omega_2\|_{L^\infty(D_L)}+
\|\zeta_1  - \zeta_2\|_{L^\infty(D_L)}\Bigr),
\end{split}
\end{equation}
where $\sigma_1>0$ is a small constant, since

\[
\|\psi_{02,\zeta_1}  - \psi_{02,\zeta_2}\|_{L^\infty(D_L)}\le C 
\|\zeta_1  - \zeta_2\|_{L^\infty(D_L)},
\]
and 

\[
\|\psi_{03,\zeta_1}  - \psi_{03,\zeta_2}\|_{L^\infty(D_L)}\le \frac{C}{|\ln\ep|}
\|\zeta_1  - \zeta_2\|_{L^\infty(D_L)}.
\]

Similarly, for $\tau_j=1$, we can prove in a similar way that 
\eqref{eq4-43a} also holds.

On the other hand, 

\begin{equation}\label{20-3-9}
\begin{split}
&\Bigl[\bigl( 2\frac{\partial \zeta_1}{\partial y_2}+\bigl(\frac{\partial (
\Gamma_0+\zeta_1)}{\partial y_2}\bigr)^2 
+
\bigl(\frac{\partial  (
\Gamma_0+\zeta_1)}{\partial y_1}\bigr)^2
\bigr)\sum\limits_{j=1}^k1_{B_\delta(x_{0,j})}
f_j\big(\mathcal{U}_{\varepsilon,\mathbf{x},\mathbf{a}}+ \psi_0
+\omega_1\big)\Bigr]\\
&-\Bigl[\bigl( 2\frac{\partial \zeta_2}{\partial y_2}+\bigl(\frac{\partial (
\Gamma_0+\zeta_2)}{\partial y_2}\bigr)^2 
+
\bigl(\frac{\partial  (
\Gamma_0+\zeta_2)}{\partial y_1}\bigr)^2
\bigr)\sum\limits_{j=1}^k1_{B_\delta(x_{0,j})}
f_j\big(\mathcal{U}_{\varepsilon,\mathbf{x},\mathbf{a}}+ \psi_0
+\omega_2\big)\Bigr]\\
=&O\Big(\frac{1
}{ |\ln\varepsilon|^{p}}\Big)\Big(\|\omega_1-\omega_2\|_{L^\infty(D_L)}+
\|\zeta_1  - \zeta_2\|_{C^1(D_L)}\Bigr).
\end{split}
\end{equation}

Combining \eqref{neq4-43a} and \eqref{20-3-9}, we obtain

\begin{equation}\label{21-3-9}
\begin{split}
  & |\ln \ep|^{p-1} 
\|R_{\varepsilon,\mathbf{x}}(\omega_1,\zeta_1)-R_{\varepsilon,\mathbf{x}}(\omega_2,\zeta_2)
\|_{L^\infty(D_L)}\\
=&O\Big(\frac{1
}{ |\ln\varepsilon|^{\sigma_1}}\Big)\Big(\|\omega_1-\omega_2\|_{L^\infty(D_L)}+
\|\zeta_1  - \zeta_2\|_{C^1(D_L)}\Bigr).
\end{split}
\end{equation}

By direct computations, we can also show

\begin{equation}\label{22-3-9}
\begin{split}
  & \|\mathcal F_1(\omega_1,\zeta_1)-\mathcal F_1(\omega_2,\zeta_2)\|_{C^{1,\beta}(0,L)}\\
=&O\Big(\frac{1
}{ |\ln\varepsilon|^{\sigma_1}}\Big)\Big(\|\omega_1-\omega_2\|_{C^{1,\beta}(0,L)}+
\|\zeta_1  - \zeta_2\|_{C^{2,\beta}(0,L)}\Bigr).
\end{split}
\end{equation}
Thus, from \eqref{21-3-9} and \eqref{22-3-9}, we prove that
 $\mathbf T$ is a contraction map.
 By the contraction mapping theorem, there exist uniquely
$\omega_{\varepsilon,\mathbf{x}}\in E_{\varepsilon,\mathbf{x}}$
and $\zeta_{\varepsilon,\mathbf{x}}$, satisfying
\[
(\omega_{\varepsilon,\mathbf{x}}, \zeta_{\varepsilon,\mathbf{x}})=\mathbf T
(\omega_{\varepsilon,\mathbf{x}}, \zeta_{\varepsilon,\mathbf{x}}).
\]
Moreover, \eqref{26-3-9} follows from \eqref{eq4-45}.

\end{proof}

\begin{remark}\label{re2-15-5}

It follows from  \eqref{220-10-5}  that for any $c$,

 \begin{equation}\label{20-15-5}
\mathcal{U}_{\varepsilon,\mathbf{x},\mathbf{a}}(y)=
\mathcal{U}_{\varepsilon,\mathbf{\hat x},\mathbf{a}}(y- ce_1),
\end{equation}
 where $\mathbf{\hat x}= \mathbf{x} -(c \boldsymbol{e}_1,\cdots, c\boldsymbol{e}_1)$, $\boldsymbol{e}_1=(1,0)$.
Then, using the uniqueness of the solution $(\omega_{\varepsilon,\mathbf{x}},
\zeta_{\varepsilon,\mathbf{x}})$, we can prove that

\begin{equation}\label{22-15-5}
\omega_{\varepsilon,\mathbf{x},\mathbf{a}}(y)=
\omega_{\varepsilon,\mathbf{\hat x},\mathbf{a}}(y-ce_1),\;\;\zeta_{\varepsilon,\mathbf{x},\mathbf{a}}(t)=
\zeta_{\varepsilon,\mathbf{\hat x},\mathbf{a}}(t-ce_1).
\end{equation}

\end{remark}

\subsection{ The case $\alpha>0$} This case is much simpler than the case $\alpha=0$,
because there is no background flow. That is, we can regard this case as a special
case in the last subsection with $\gamma=0$. The solution 
we look has the form
\[
u_\varepsilon=\mathcal{U}_{\varepsilon,\mathbf{x},\mathbf{a}}
+\omega_{\varepsilon,\mathbf{x}},\quad \Gamma_2=\zeta.
\]
Then we have the same estimates stated in Lemmas~\ref{l2-12-5} and \ref{l3-12-5}.

The operator is defined as

 \begin{equation}\label{10-4-9}
\begin{split}
&\hat{\mathbf T}(\omega,\zeta)= (\hat{\mathbf T}_1(\omega,\zeta), \hat{\mathbf T}_2(\omega,\zeta))\\
\\
=&\hat  T\Bigl( Q_{\varepsilon,\mathbf{x}}
\bigl(l_{\varepsilon,\mathbf{x}}
+R_{\varepsilon,\mathbf{x}}(\omega,\zeta)\bigr),\;  \Theta( \omega, \zeta)-  \frac1L \int_{0}^L
 \Theta( \omega, \zeta)\Bigr),
\end{split}
\end{equation}
where $\hat  T$ is given in \eqref{11-4-9}, and

\begin{equation}\label{20-4-9}
\begin{split}
 \Theta( \omega, \zeta)= &\alpha^2\Bigl[\frac1{ 1+2\frac{\partial \zeta}{\partial y_2}+\bigl(\frac{\partial \zeta}{\partial y_2}\bigr)^2 +
\bigl(\frac{\partial \zeta}{\partial y_1}\bigr)^2  } -1
 \Bigr]\frac{\partial^2 \zeta}{\partial^2 y_1}
  \\
 & +\frac{\alpha^2\Bigl[\frac{\partial^2 \zeta}{\partial^2 y_1} \frac{\partial \zeta}{\partial y_2}
-\frac{\partial^2 \zeta}{\partial y_1\partial y_2} \frac{\partial \zeta}{\partial y_1}\Bigr]
 }{ 1+2\frac{\partial \zeta}{\partial y_2}+\bigl(\frac{\partial \zeta}{\partial y_2}\bigr)^2 +
\bigl(\frac{\partial \zeta}{\partial y_1}\bigr)^2  }
\\
&-\frac12\frac{2\frac{\partial \zeta}{\partial y_2}+\bigl(\frac{\partial \zeta}{\partial y_2}\bigr)^2 +
\bigl(\frac{\partial \zeta}{\partial y_1}\bigr)^2
 }{ 1+2\frac{\partial \zeta}{\partial y_2}+\bigl(\frac{\partial \zeta}{\partial y_2}\bigr)^2 +
\bigl(\frac{\partial \zeta}{\partial y_1}\bigr)^2  }\Bigl(
\frac{\partial [ \mathcal{U}_{\varepsilon,\mathbf{x},\mathbf{a}}
+\omega  ] } {\partial y_2}
\Bigr)^2\Bigr|_{y_2=\pi}\\
&
-\frac12\Bigl(
\frac{\partial [ \mathcal{U}_{\varepsilon,\mathbf{x},\mathbf{a}}
+\omega  ] } {\partial y_2}
\Bigr)^2
+\frac12  \Bigl(
\frac{\partial  \mathcal{U}_{\varepsilon,\mathbf{x},\mathbf{a}}
 } {\partial y_2}
\Bigr)^2.
\end{split}
\end{equation}
Note that on $y_2=\pi$, $\Theta( \omega, \zeta)$ is higher order small term than
$\omega$ and $\zeta$.  Similar to 
Proposition~\ref{pro4-4}, we can prove the following result.

\begin{proposition}\label{p10-4-9}
There is an $\varepsilon_0>0$ such that for any $\varepsilon\in (0,\varepsilon_0)$ and  $\mathbf{x} \in \prod_{j=1}^k B_\delta(x_{0,j})$, there exist uniquely
$\omega_{\varepsilon,\mathbf{x}}\in E_{\varepsilon,\mathbf{x}}$
and $\zeta_{\varepsilon,\mathbf{x}}$, satisfying
\[
(\omega_{\varepsilon,\mathbf{x}}, \zeta_{\varepsilon,\mathbf{x}})=\hat{\mathbf T}
(\omega_{\varepsilon,\mathbf{x}}, \zeta_{\varepsilon,\mathbf{x}}).
\]
Moreover, 
\begin{equation}\label{26-3-9}
\|\omega_{\varepsilon,\mathbf{x}}\|_{L^\infty(\Omega)}+
\| \omega_{\varepsilon,\mathbf{x}}\|_{C^{2,\beta}(D_L\setminus \cup_{j=1}^k B_{\delta}(x_j))}
+\|\zeta_{\varepsilon,\mathbf{x}}\|_{C^{2,\beta}(0,L)   }\leq  \frac{C}{|\ln\varepsilon|^{1+\sigma}}.
\end{equation}

\end{proposition}

\section{existence of solutions}

In Section~5, we prove the existence of
$\left(\omega_{\varepsilon,\mathbf{x}}, \Gamma_{\varepsilon,\mathbf{x}} \right)\in E_{\varepsilon,\mathbf{x}} \times C^{2,\beta}$, such that
\begin{equation}\label{eq4-1-1}
\begin{split}
-&\Delta (\mathcal{U}_{\varepsilon,\mathbf{x},\mathbf{a}}+\omega_{\varepsilon,\mathbf{x}})-
\frac{1}{\varepsilon^2}\sum\limits_{j=1}^k1_{B_\delta(x_{0,j})}
 |1+\Gamma_{\varepsilon,\mathbf{x}}'|^2f_j\big(\mathcal{U}_{\varepsilon,\mathbf{x},\mathbf{a}}+\psi_0
 +\omega_{\varepsilon,\mathbf{x}}\big)
\\
&=\sum\limits_{j=1}^k\sum\limits_{h=1}^2\varepsilon^2b_{jh}Z_{\varepsilon,x_j,h},
\end{split}
\end{equation}

In this section, we will choose $\mathbf{x}$ suitably such that the corresponding $b_{jh}$ is zero for $j=1,\cdots,k$,\,\, $h=1,2$. Thus $(U_{\ep,\mathbf{x},\mathbf{a}} + \omega_{\ep,\mathbf{x}}, \Gamma_{\ep,\mathbf{x}})$ is a  solution of the water wave equations.
 We will prove the following result.

\begin{theorem}\label{th1-4-9}

Suppose that $\tilde{\mathcal K}_p(\mathbf{\hat x})$ has a non-degenerate critical point
$\mathbf{\hat x}_0$, where $\mathbf{\hat x}= (\hat x_1,\cdots, \hat x_k)$,
 with $\hat x_{j_0}=(0, \hat x_{j_02})$,
  then there exists $\mathbf{\hat x}_{\ep}$, such that
  $b_{jh}$ is zero for $j=1,\cdots,k$,\,\, $h=1,2$. Moreover, $\mathbf{\hat x}_{\ep}\to
  \mathbf{\hat x}_0
  $ as $\ep\to 0$.

\end{theorem}

To prove Theorem~\ref{th1-4-9}, we need the following result.

\begin{proposition}\label{t1-15-5}
If $\mathbf{x}$ is a solution of the following equation

\begin{equation}\label{20-27-8}
\begin{split}
&\int_{D_L}\Bigl( \nabla (\mathcal{U}_{\varepsilon,\mathbf{x},\mathbf{a}}+\omega_{\varepsilon,\mathbf{x}})
\nabla ( \frac{\partial \mathcal{U}_{\varepsilon,\mathbf{x},\mathbf{a}} }{\partial x_{ml}} +
\frac{\partial  \omega_{\varepsilon,\mathbf{x}}
}{\partial x_{ml}} )\\
&-
\frac{1}{\varepsilon^2}\sum\limits_{j=1}^k1_{B_\delta(x_{0,j})}
\int_{D_L}\, |1+\Gamma'|^2f_j\big(\mathcal{U}_{\varepsilon,\mathbf{x},\mathbf{a}}
 +\psi_0+\omega_{\varepsilon,\mathbf{x}}\big)\bigl(\frac{\partial \mathcal{U}_{\varepsilon,\mathbf{x},\mathbf{a}} }{\partial x_{ml}} +
\frac{\partial  \omega_{\varepsilon,\mathbf{x}}
}{\partial x_{ml}} \bigr)\\
=&0,\quad m=1,\cdots,k, \; l=1,2
\end{split}
\end{equation}
 then $b_{jh}=0$, $j=1,\cdots,k$, $h=1,2$.
\end{proposition}

\begin{proof}

If $\mathbf{x}$ is a solution of \eqref{20-27-8}, then using \eqref{eq4-1-1}, we obtain

\begin{equation}\label{30-15-5}
\begin{split}
\sum\limits_{j=1}^k\sum\limits_{h=1}^2\varepsilon^2b_{jh}
\int_{D_L}\,Z_{\varepsilon,x_j,h}
\bigl(\frac{\partial \mathcal{U}_{\varepsilon,\mathbf{x},\mathbf{a}} }{\partial x_{ml}} +
\frac{\partial  \omega_{\varepsilon,\mathbf{x}}
}{\partial x_{ml}} \bigr)=0.
\end{split}
\end{equation}

From

\[
\int_{D_L}\, Z_{\varepsilon,\mathbf{x},j,h}\omega=0,  \quad \forall\; \mathbf x,
\]
we  find
\[
\int_{D_L}\, Z_{\varepsilon,\mathbf{x},j,h}
\frac{\partial  \omega_{\varepsilon,\mathbf{x}}
}{\partial x_{ml}} =-\int_{D_L}\,
\frac{\partial  Z_{\varepsilon,\mathbf{x},j,h}
}{\partial x_{ml}} \omega_{\varepsilon,\mathbf{x}}.
\]
Thus, we have

\begin{equation}\label{34-15-5}
\sum\limits_{j=1}^k\sum\limits_{h=1}^2\varepsilon^2b_{jh}
\int_{D_L}\,\Bigl(Z_{\varepsilon,x_j,h}
\frac{\partial \mathcal{U}_{\varepsilon,\mathbf{x},\mathbf{a}} }
{\partial x_{ml}} +
\frac{\partial Z_{\varepsilon,x_j,h} }
{\partial x_{ml}}  \omega_{\varepsilon,\mathbf{x}}\Bigr)=0.
\end{equation}

In view of

\begin{equation}\label{35-15-5}
\begin{split}
\varepsilon^2\int_{D_L}\,Z_{\varepsilon,\mathbf{x},j,h}V_{\varepsilon,\mathbf{x},
i,l}=&p\int_{D_L}\, (U_{\varepsilon,x_j,a_{\varepsilon,j}}-a_{\varepsilon,j})_+^{p-1}(\frac{\partial U_{\varepsilon,x_j,a_{\varepsilon,j}}}{\partial x_{jh}}-\frac{\partial a_{\varepsilon,j}}{\partial x_{jh}})V_{\varepsilon,\mathbf{x},i,l}
\\
=&\delta_{ij}\delta_{hl}\frac{c}{|\ln\varepsilon|^{p+1}}+O(\frac{s_{\varepsilon,1}}{|\ln \varepsilon|^{p+1}}),
\end{split}
\end{equation}
where $c>0$ is a constant, $\delta_{ij}=1$ if $i=j$, $\delta_{ij}=0$ if $i\neq j$,
and

\begin{equation}\label{36-15-5}
\begin{split}
&\varepsilon^2\int_{D_L}\,\frac{\partial Z_{\varepsilon,x_j,h} }
{\partial x_{ml}}  \omega_{\varepsilon,\mathbf{x}}\\
=&p\int_{D_L}\, (U_{\varepsilon,x_j,a_{\varepsilon,j}}-a_{\varepsilon,j})_+^{p-1}
\frac{\partial^2 [ (U_{\varepsilon,x_j,a_{\varepsilon,j}}-a_{\varepsilon,j})_+] }{\partial x_{jh}\partial x_{ml}}\omega_{\varepsilon,\mathbf{x}}
\\
&+p(p-1)\int_{D_L}\, (U_{\varepsilon,x_j,a_{\varepsilon,j}}-a_{\varepsilon,j})_+^{p-2}
\frac{\partial [ (U_{\varepsilon,x_j,a_{\varepsilon,j}}-a_{\varepsilon,j})_+] }{\partial x_{jh}}
\frac{\partial [ (U_{\varepsilon,x_j,a_{\varepsilon,j}}-a_{\varepsilon,j})_+] }{\partial x_{ml}}\omega_{\varepsilon,\mathbf{x}}
\\
=&o\bigl(\frac{1}{|\ln\varepsilon|^{p+1}}\bigr),
\end{split}
\end{equation}
we can solve \eqref{34-15-5} to get  $b_{jh}=0$.

\end{proof}

To solve \eqref{20-27-8}, we will use
 Remark~\ref{re2-15-5}.  Fix $j_0$.  We let  $\mathbf{\hat x}= (\hat x_1,\cdots, \hat x_k)$,
 with $\hat x_{j_0}=(0, \hat x_{j_02})$.
 Suppose that $\mathbf{\hat x}$ is a solution of the following equation

\begin{equation}\label{21-27-8}
\begin{split}
&\int_{D_L}\Bigl( \nabla (\mathcal{U}_{\varepsilon,\hat{\mathbf{x}},\mathbf{a}}+\omega_{\varepsilon,\hat{\mathbf{x}}})
\nabla ( \frac{\partial \mathcal{U}_{\varepsilon,\hat{\mathbf{x}},\mathbf{a}} }{\partial \hat x_{ml}} +
\frac{\partial  \omega_{\varepsilon,\hat{\mathbf{x}}}
}{\partial \hat x_{ml}} )\\
&-
\frac{1}{\varepsilon^2}\sum\limits_{j=1}^k1_{B_\delta(x_{0,j})}
\int_{D_L}\, |1+\Gamma'|^2f_j\big(\mathcal{U}_{\varepsilon,\hat{\mathbf{x}},\mathbf{a}}
+\psi_0 +\omega_{\varepsilon,\hat{\mathbf{x}}}\big)\bigl(\frac{\partial \mathcal{U}_{\varepsilon,\hat{\mathbf{x}},\mathbf{a}} }{\partial \hat x_{ml}} +
\frac{\partial  \omega_{\varepsilon,\hat{\mathbf{x}}}
}{\partial \hat x_{ml}} \bigr)\\
=&0,
\end{split}
\end{equation}
for $m\ne j_0$, $l=1, 2$  and $m=j_0$, $l=2$. Or equivalently,

\begin{equation}\label{1-11-9}
\begin{split}
&\int_{D_L}\Bigl( \nabla (\mathcal{U}_{\varepsilon,\mathbf{x}-
\mathbf x_{j_0},\mathbf{a}}+\omega_{\varepsilon,\mathbf{x}-\mathbf x_{j_0}})
\nabla ( \frac{\partial \mathcal{U}_{\varepsilon,\mathbf{x}-\mathbf x_{j_0},\mathbf{a}}}{\partial  x_{ml}} +
\frac{\partial  \omega_{\varepsilon,\mathbf{x}-\mathbf x_{j_0}}}
{\partial  x_{ml}}  )\\
&-
\frac{1}{\varepsilon^2}\sum\limits_{j=1}^k1_{B_\delta(x_{0,j})}
\int_{D_L}\, |1+\Gamma'|^2f_j\big(\mathcal{U}_{\varepsilon,\mathbf{x}-
\mathbf x_{j_0},\mathbf{a}}
+\psi_0 +\omega_{\varepsilon,\mathbf{x}-
\mathbf x_{j_0}}\big)\bigl(\frac{\partial \mathcal{U}_{\varepsilon,\mathbf{x}-
\mathbf x_{j_0},\mathbf{a}} }{\partial  x_{ml}} +
\frac{\partial  \omega_{\varepsilon,\mathbf{x}-
\mathbf x_{j_0}}
}{\partial  x_{ml}} \bigr)\\
=&0,
\end{split}
\end{equation}
for $\mathbf{ x}= ( x_1,\cdots,  x_k)$, satisfying $\mathbf{x}-
\mathbf x_j=\mathbf{\hat x}$, 
$m\ne j_0$, $l=1, 2$  and $m=j_0,$, $l=2$, where $\mathbf x_{j_0}=(x_{j_01}\boldsymbol{e}_1,\cdots, x_{j_01}\boldsymbol{e}_1)$.
In view of

\[
\begin{split}
\frac{\partial \mathcal{U}_{\varepsilon,\mathbf{x}-
\mathbf x_{j_0},\mathbf{a}} }{\partial  x_{j_0l}} +
\frac{\partial  \omega_{\varepsilon,\mathbf{x}-
\mathbf x_{j_0}}
}{\partial  x_{j_0l}} 
=-\sum_{m\ne j_0}\Bigl(
\frac{\partial \mathcal{U}_{\varepsilon,\mathbf{x}-
\mathbf x_{j_0},\mathbf{a}} }{\partial  x_{ml}} +
\frac{\partial  \omega_{\varepsilon,\mathbf{x}-
\mathbf x_{j_0}}
}{\partial  x_{ml}}\Bigr)
\end{split}
\]
we see that \eqref{1-11-9} holds for $m=j_0$, $l=1$.
By \eqref{20-15-5}  and \eqref{22-15-5},  $\mathbf{x}$ is a solution  of \eqref{20-27-8}.

We now solve \eqref{21-27-8}. For this purpose, we need to expand the left hand side
of \eqref{21-27-8} to find the main term.  This is not easy by direct computations, 
because the term $\frac{\partial \omega}{\partial x_{jh}}$ is not small in $D_L$,
though it is small away from the concentration points.  We need to proceed indirectly
by using the equation \eqref{eq4-1-1}.

\begin{proposition}\label{p2-22-5}
We have

\[
\begin{split}
\text{LHS of \eqref{21-27-8}}
=\frac{2\pi^2}{|\ln\varepsilon|^2}\frac{\partial \tilde{\mathcal K}_p(\mathbf{\hat x}) }{\partial  x_{ml} }
+o\bigl(\frac1{|\ln \ep|^{2}}\bigr).
\end{split}
\]

\end{proposition}

\begin{proof}

Using \eqref{eq4-1-1}, we obtain

\begin{equation}\label{1-21-5}
\begin{split}
\text{LHS of \eqref{21-27-8}}=&
\sum\limits_{j=1}^k\sum\limits_{h=1}^2\varepsilon^2b_{jh}\int_{D_L}Z_{\varepsilon,x_j,h}
\frac{\partial (\mathcal{U}_{\varepsilon,\mathbf{x},
\mathbf{a}}+\omega_{\varepsilon,\mathbf{x}})}{\partial  x_{ml} }.
\end{split}
\end{equation}

 We have

 \[
 \begin{split}
& \frac{\partial \mathcal{U}_{\varepsilon,\mathbf{x},
\mathbf{a}}}{\partial  x_{ml} }=\delta_{jm} (-1)^{\tau_m}
\frac{\partial PU_{\varepsilon,x_m,a_m}}{\partial  x_{ml} }+O\bigl(\frac1{|\ln \ep|}\bigr)\\
=&\delta_{jm} (-1)^{\tau_m}
\frac{\partial U_{\varepsilon,x_m,a_m}}{\partial  x_{ml} }+O\bigl(\frac1{|\ln \ep|}\bigr)\\
=&
-(-1)^{\tau_m}
\frac{\partial U_{\varepsilon,x_m,a_m}}{\partial  y_{l} }+O\bigl(\frac1{|\ln \ep|}\bigr).
\end{split}
\]
On the other hand,  from

\[
 \int_{D_L}\, Z_{\varepsilon,\mathbf{x},j,h}\omega_{\varepsilon,\mathbf{x}}=0, \quad \forall\; x,
 \]
we obtain

\[
\begin{split}
& \int_{D_L}\, Z_{\varepsilon,\mathbf{x},j,h}\frac{\partial \omega_{\varepsilon,\mathbf{x}}}{\partial  x_{ml} }
 =-\int_{D_L}\,\frac{\partial  Z_{\varepsilon,\mathbf{x},j,h}}{\partial  x_{ml} }
 \omega_{\varepsilon,\mathbf{x}}\\
 =&-\delta_{jm} \int_{D_L}\,\frac{\partial  Z_{\varepsilon,\mathbf{x},m,h}}{\partial  x_{ml} }
 \omega_{\varepsilon,\mathbf{x}}+O\bigl(\frac1{|\ln \ep|}\bigr)\|\omega_{\varepsilon,\mathbf{x}}\|_{
 L^\infty(D_L)}\\
 =&\delta_{jm} \int_{D_L}\,\frac{\partial  Z_{\varepsilon,\mathbf{x},m,h}}{\partial  y_{l} }
 \omega_{\varepsilon,\mathbf{x}}+O\bigl(\frac1{|\ln \ep|}\bigr)\|\omega_{\varepsilon,\mathbf{x}}\|_{
 L^\infty(D_L)}\\
 =&- \int_{D_L}\, Z_{\varepsilon,\mathbf{x},m,h}\frac{\partial \omega_{\varepsilon,\mathbf{x}}}{\partial  y_{l} }
 +O\bigl(\frac1{|\ln \ep|}\bigr)\|\omega_{\varepsilon,\mathbf{x}}\|_{
 L^\infty(D_L)}.
 \end{split}
 \]
So we have proved that

\begin{equation}\label{6-21-5}
\begin{split}
&\text{ RHS of \eqref{1-21-5}}\\
=&
-\sum\limits_{h=1}^2\varepsilon^2b_{mh}\Bigl(\int_{D_L}
 Z_{\varepsilon,\mathbf{x},m,h}\Bigl[\frac{\partial (-1)^{\tau_m} U_{\varepsilon,x_m,a_m}}{\partial  y_{l} }+\frac{\partial \omega_{\varepsilon,\mathbf{x}}}{\partial  y_{l} }\Bigr]+O\bigl(\frac1{|\ln \ep|}\bigr)
\Bigr)\\
=&-\sum\limits_{h=1}^2\varepsilon^2b_{mh}(\int_{B_\delta(x_m)}
 Z_{\varepsilon,\mathbf{x},m,h}\Bigl[\frac{\partial (-1)^{\tau_m} U_{\varepsilon,x_m,a_m}}{\partial  y_{l} }+\frac{\partial \omega_{\varepsilon,\mathbf{x}}}{\partial  y_{l} }\Bigr]+O\bigl(\frac1{|\ln \ep|}\bigr)\ep^2|b_{mh}|\\
 =&-\sum\limits_{h=1}^2\varepsilon^2b_{mh}(\int_{B_\delta(x_m)}
 Z_{\varepsilon,\mathbf{x},m,h}\frac{\partial (
 \mathcal{U}_{\varepsilon,\mathbf{x},
\mathbf{a}}+\omega_{\varepsilon,\mathbf{x}})}{\partial  y_{l} }+O\bigl(\frac1{|\ln \ep|}\bigr)\ep^2|b_{mh}|.
\end{split}
\end{equation}

  Denote $u_{\varepsilon,\mathbf{x}} =\mathcal{U}_{\varepsilon,\mathbf{x},
\mathbf{a}}+\omega_{\varepsilon,\mathbf{x}}$. 
 From \eqref{eq4-1-1}, we find
 
\begin{equation}\label{26-4-9}
\begin{split}
&\int_{B_\delta(x_m)}  \Delta u_{\varepsilon,\mathbf{x}}  \frac{\partial u_{\varepsilon,\mathbf{x}}
}{\partial  y_{l} }-\frac1{\ep^2}
\int_{B_\delta(x_m)}
 |1+\Gamma'|^2f_m(u_{\varepsilon,\mathbf{x}}+\psi_0)\frac{\partial u_{\varepsilon,\mathbf{x}}
}{\partial  y_{l} }
\\
=&\sum\limits_{h=1}^2\varepsilon^2b_{mh}\int_{B_\delta(x_m)}
 Z_{\varepsilon,\mathbf{x},m,h}\frac{\partial (
 \mathcal{U}_{\varepsilon,\mathbf{x},
\mathbf{a}}+\omega_{\varepsilon,\mathbf{x}})}{\partial  y_{l} }
\end{split}
\end{equation}

Then, combining
\eqref{1-21-5}, \eqref{6-21-5}  and \eqref{26-4-9}, we obtain

\begin{equation}\label{7-21-5}
\begin{split}
&\text{LHS of \eqref{21-27-8}}\\
=&\int_{B_\delta(x_m)}  \Delta u_{\varepsilon,\mathbf{x}}  \frac{\partial u_{\varepsilon,\mathbf{x}}
}{\partial  y_{l} }-\frac1{\ep^2}
\int_{B_\delta(x_m)}
 |1+\Gamma'|^2f_m(u_{\varepsilon,\mathbf{x}}+\psi_0)\frac{\partial u_{\varepsilon,\mathbf{x}}
}{\partial  y_{l} }
+O\bigl(\frac1{|\ln \ep|}\bigr)\ep^2|b_{mh}|.
\end{split}
\end{equation}
We also have

\begin{equation}\label{8-21-5}
\begin{split}
&\int_{B_\delta(x_m)}
 |1+\Gamma'|^2f_m(u_{\varepsilon,\mathbf{x}}+\psi_0)\frac{\partial u_{\varepsilon,\mathbf{x}}
}{\partial  y_{l} }\\
=&-\int_{B_\delta(x_m)}
 |1+\Gamma'|^2f_m(u_{\varepsilon,\mathbf{x}}+\psi_0)\frac{\partial \psi_0
}{\partial  y_{l} }-\int_{B_\delta(x_m)} F_m(u_{\varepsilon,\mathbf{x}}+\psi_0) \frac{\partial |1+\Gamma'|^2
}{\partial  y_{l} }
\\
=&- \frac{2\pi^2\ep^2}{|\ln\varepsilon|^2} (-1)^{\tau_m} \kappa_m\frac{\partial \psi_{01}
}{\partial  y_{l} }
+O\bigl(\frac1{|\ln \ep|^{2+\sigma}}\bigr)\ep^2,
\end{split}
\end{equation}
where $F_m(t)=\int_0^t f_m(s)\,ds$.

 For the coefficient $b_{jh}$ in  \eqref{eq4-1-1},
similar to \eqref{eq4-20} and  \eqref{eq4-22}, we can deduce
that

\begin{equation}\label{9-21-5}
b_{jh}= O\bigl( \frac{s_{\ep i}}{\ep^2|\ln s_{\ep, i}|} \bigr).
\end{equation}

Combining \eqref{7-21-5}-\eqref{9-21-5}, we are led to

\begin{equation}\label{nn10-21-5}
\begin{split}
&\text{LHS of \eqref{21-27-8}}
=\int_{B_\delta(x_m)}  \Delta u_{\varepsilon,\mathbf{x}}  \frac{\partial u_{\varepsilon,\mathbf{x}}
}{\partial  y_{l} }+O\bigl(\frac1{|\ln \ep|^{2+\sigma}}\bigr)\\
=& \int_{\partial B_\delta(x_m)} \frac{\partial  u_{\varepsilon,\mathbf{x}} }{\partial \nu}
\frac{\partial  u_{\varepsilon,\mathbf{x}} }{\partial y_l}-\frac12
\int_{\partial B_\delta(x_m)} |\nabla u_{\varepsilon,\mathbf{x}}|^2\nu_j
 -\frac{2\pi^2}{|\ln\varepsilon|^2} (-1)^{\tau_m} \kappa_m\frac{\partial \psi_{01}
}{\partial  y_{l} }
+O\bigl(\frac1{|\ln \ep|^{2+\sigma}}\bigr),
\end{split}
\end{equation}
where $\nu$ is the outward unit normal of $\partial B_\delta(x_m)$.

On the other hand, we have

\[
 u_{\varepsilon,\mathbf{x}} =\sum_{j=1}^k (-1)^{\tau_j}
 \frac{2\hat\kappa_j\pi}{\ln s_{\varepsilon,j}} G_p(y,x_j)+ O\bigl(\frac1{|\ln \ep|^{1+\sigma}}\bigr),\quad \forall\; \text{in}\;   C^1( \partial B_\delta(x_m)),
 \]
 where $\hat\kappa_j=  \kappa_j-(-1)^{\tau_j} \psi_{01}(x_{0,j}) $, $\psi_{01}(y) =\gamma\Bigl(\frac{\pi^2}{8} -\frac12\big( y_2-\frac\pi2\bigr)^2\Bigr)$.
 As a result,

 \begin{equation}\label{10-21-5}
\begin{split}
& \int_{\partial B_\delta(x_m)} \frac{\partial  u_{\varepsilon,\mathbf{x}} }{\partial \nu}
\frac{\partial  u_{\varepsilon,\mathbf{x}} }{\partial y_l}-\frac12
\int_{\partial B_\delta(x_m)} |\nabla u_{\varepsilon,\mathbf{x}}|^2\nu_j
\\
=& \int_{\partial B_\delta(x_m)} \frac{\partial  G^*(y)}{\partial \nu}
\frac{\partial  G^*(y)}{\partial y_l}-\frac12
\int_{\partial B_\delta(x_m)} |\nabla G^*(y)|^2\nu_j
+O\bigl(\frac1{|\ln \ep|^{2+\sigma}}\bigr)\\
=&\frac{2\pi^2}{|\ln\varepsilon|^2}\frac{\partial \mathcal K_p(\mathbf{x}) }{\partial  x_{ml} }
+o\bigl(\frac1{|\ln \ep|^{2}}\bigr).
\end{split}
\end{equation}
where

\[
G^*(y)=\sum_{j=1}^k (-1)^{\tau_j}
 \frac{2\hat\kappa_j\pi}{\ln s_{\varepsilon,j}} G_p(y,x_j).
 \]
Thus the result follows from \eqref{nn10-21-5} and \eqref{10-21-5}.

 \end{proof}

\begin{proof}[Proof of Theorem~\ref{th1-4-9}]
This is a direct consequence of Propositions~\ref{t1-15-5} and \ref{p2-22-5}.
\end{proof}
\medskip

Now we are in the position to prove the main result.
\begin{proof}[Proof of Theorem \ref{th2-1}]
By the result in Appendix~B,
$\widetilde{\mathcal K}_p(\hat{x})$ possesses a "non-degenerate"
asymmetric critical points. Subsequently,
by Theorem \ref{th1-4-9}, there exists $\ep_0>0$ such that
for any $0<\ep<\ep_0$, there exists a solution of steady
water wave equations \eqref{n10-24-8}-\eqref{21-24-8}
represented as
\[
u_\ep=\mathcal{U}_{\ep,\mathbf{x_\ep},\mathbf{a_\ep}}
+\omega_{\ep,\mathbf{x_\ep}},
\qquad
\Gamma_s=\Gamma_0(\cdot,\pi)+\zeta_\ep,
\]
where $\mathbf{x_\ep}$ is close the the asymmetric critical point
of $\widetilde{\mathcal K}_p(\hat{x})$.

We note that the  asymmetry of the main term $\mathcal{U}_{\varepsilon,\mathbf{x},\mathbf{a}}$ 
follows from the 
 asymmetry of the    critical point
 of $\tilde{\mathcal K}_p(\mathbf{x})$, together with the fact that the fuid velocity is nowhere zero on the free surface.  This
 results in the asymmetry of $u_\ep$.
 
 On the other hand, $\Gamma_s$ can not be symmetric. Suppose this
 is not true. Then
 its harmonic is symmetric.  This gives that the left hand side of \eqref{62-24-8}
 is symmetric. But the right hand side of \eqref{62-24-8}
 is asymmetric since $u_\ep$ is asymmetric.  This is a contradiction.

\end{proof}

\medskip
\noindent\textbf{Acknowledgments.}
We thank Adrian Constantin for helpful discussions during the third
author's visit to the University of Vienna. We also thank Chongchun
Zeng for helpful discussions during his visit to the Chern Institute
of Mathematics at Nankai University.
S.Yan is supported by  NSFC (No.
12631005). Y. Wang is supported by NSFC (No.12471110).
\section*{Statements and Declarations}

\noindent
\textbf{Competing Interests.}
The authors declare that they have no competing interests.

\medskip
\noindent
\textbf{Data Availability.}
Data availability is not applicable to this article, as no
datasets were generated or analyzed during the current study.

\medskip
\noindent
\textbf{Author Contributions.}
All authors contributed to the research and preparation of
the manuscript. All authors read and approved the final manuscript.

\medskip
\noindent
\textbf{Declaration of AI-assisted tools.}
The authors used ChatGPT (OpenAI) solely for language
polishing.
All mathematical results, proofs, and conclusions
are the original work of the authors.
All AI-assisted materials were reviewed and verified
by the authors, who take full responsibility for
the content of the manuscript.

\appendix

\section{The nonlocal mixed boundary value problem}

In this section, we study problem \eqref{21-24-8}. Let
\[
\theta=\frac{4(g+\frac{\pi}{2}\gamma^2)}{\pi^2\gamma^2}>0,
\quad \gamma\ne0.
\]
We consider the following nonlocal mixed boundary value problem
\begin{equation}\label{2-19-6}
\begin{cases}
\Delta w=0,\quad \text{in}\; D_L,\\
\text{$w$ is periodic in $y_1$,}\\
w(y_1,0)=0,\\
\displaystyle
\frac{\partial w}{\partial y_2}
-\frac1L\int_0^L
\frac{\partial w}{\partial y_2}(t,\pi)\,dt
-\theta w=f,\quad \text{on}\; \{y_2=\pi\}.
\end{cases}
\end{equation}

\subsection{Solvability}

We study the solvability of \eqref{2-19-6}. The method is standard.
We give a brief discussion for completeness.

Since $w$ is periodic in $y_1$, we write
\begin{equation}\label{3-19-6}
w=w_0(y_2)+\sum_{n=1}^{+\infty}
\Bigl(
w_{n,1}(y_2)\cos\frac{2n\pi y_1}{L}
+w_{n,2}(y_2)\sin\frac{2n\pi y_1}{L}
\Bigr),
\end{equation}
and
\begin{equation}\label{n3-19-6}
f=f_0+\sum_{n=1}^{+\infty}
\Bigl(
f_{n,1}\cos\frac{2n\pi y_1}{L}
+f_{n,2}\sin\frac{2n\pi y_1}{L}
\Bigr).
\end{equation}
Then
\begin{equation}\label{4-19-6}
w''_0=0,\quad
w''_{n,i}-\Bigl(\frac{2n\pi}{L}\Bigr)^2w_{n,i}=0,
\;\; i=1,2,\;n=1,\cdots,
\end{equation}
\begin{equation}\label{0-19-6}
w_0(0)=0,\quad
\frac{\partial w_0}{\partial y_2}
-\frac1L\int_0^L
\frac{\partial w_0}{\partial y_2}\,dy_1
-\theta w_0=f_0,
\;\;\text{on}\;\{y_2=\pi\},
\end{equation}
and for $n\ge1$, $i=1,2$,
\begin{equation}\label{n0-19-6}
w_{n,i}(0)=0,\quad
\frac{\partial w_{n,i}}{\partial y_2}
-\theta w_{n,i}=f_{n,i},
\;\;\text{on}\;\{y_2=\pi\},
\end{equation}

For $n=0$, we have $w_0=a y_2$ and $-\theta a\pi=f_0$.
Then $w_0=-\frac{f_0}{\theta\pi}y_2$.

For $n\ge1$, from $w_{n,i}(0)=0$ we obtain
\begin{equation}\label{5-19-6}
w_{n,i}=a_{n,i}\sinh\Bigl(\frac{2n\pi}{L}y_2\Bigr),
\end{equation}
and
\begin{equation}\label{n5-19-6}
a_{n,i}
\Bigl(
\frac{2n\pi}{L}\cosh\frac{2n\pi^2}{L}
-\theta\sinh\frac{2n\pi^2}{L}
\Bigr)=f_{n,i}.
\end{equation}
Thus, if
\[
\theta\ne\frac{2n\pi}{L}\coth\frac{2n\pi^2}{L},
\]
then
\[
a_{n,i}
=
\frac{f_{n,i}}
{\frac{2n\pi}{L}\cosh\frac{2n\pi^2}{L}
-\theta\sinh\frac{2n\pi^2}{L}}.
\]

\begin{proposition}\label{p1-25-8}
If $\theta>0$ and
$\theta\ne\frac{2n\pi}{L}\coth\frac{2n\pi^2}{L}$
for $n\ge1$, then \eqref{2-19-6} is uniquely solvable.
\end{proposition}

\subsection{Schauder estimate}

We derive the Schauder estimate for $w$ in terms of $f$.
Let $\beta\in (0, 1)$.

First, from Theorems~6.26 and 6.2 and Lemma~6.4 in \cite{GT},
we obtain
\begin{equation}\label{1-20-6}
\|w\|_{C^{2,\beta}(D_L)}
\le C\|w\|_{L^\infty(D_L)}
+C\left\|
f+\frac1L\int_0^L
\frac{\partial w}{\partial y_2}(t,\pi)\,dt
\right\|_{C^{1,\beta}(0,L)}.
\end{equation}
We also have
\[
\left\|
\frac1L\int_0^L
\frac{\partial w}{\partial y_2}(t,\pi)\,dt
\right\|_{C^{1,\beta}(0,L)}
=
\left\|
\frac1L\int_0^L
\frac{\partial w}{\partial y_2}(t,\pi)\,dt
\right\|_{L^\infty(0,L)}
\le \|w\|_{C^1(D_L)},
\]
which, together with \eqref{1-20-6}, gives
\begin{equation}\label{2-20-6}
\|w\|_{C^{2,\beta}(D_L)}
\le C\|w\|_{C^1(D_L)}
+C\|f\|_{C^{1,\beta}(0,L)}.
\end{equation}

\begin{proposition}\label{p2-25-8}
If $\theta>0$ and
$\theta\ne\frac{2n\pi}{L}\coth\frac{2n\pi^2}{L}$
for $n\ge1$, then
\begin{equation}\label{n2-20-6}
\|w\|_{C^{2,\beta}(D_L)}
\le C\|f\|_{C^{1,\beta}(0,L)}.
\end{equation}
\end{proposition}

\begin{proof}
From \eqref{2-20-6}, it is sufficient to prove
\begin{equation}\label{3-20-6}
\|w\|_{C^1(D_L)}
\le C\|f\|_{C^{1,\beta}(0,L)}.
\end{equation}

Suppose that there are $f_n$ with
$\|f_n\|_{C^{1,\beta}(0,L)}\to 0$ and $w_n$ with
$\|w_n\|_{C^1(D_L)}=1$, satisfying \eqref{2-19-6}.
From \eqref{2-20-6}, we obtain
\[
\|w_n\|_{C^{2,\beta}(D_L)}\le C.
\]
Up to a subsequence, $w_n\to w$ in $C^2(\overline{D_L})$, and the limit
satisfies \eqref{2-19-6} with $f=0$.
Since $\theta\ne0$ and
$\theta\ne\frac{2n\pi}{L}\coth\frac{2n\pi^2}{L}$
for $n\ge1$, Proposition~\ref{p1-25-8} gives $w=0$.
This contradicts to $\|w\|_{C^1(D_L)}=1$.
\end{proof}

\section{critical points of the K.-R. functions}

First, let us recall that the function $\tilde{\mathcal K}_p$ is defined in \eqref{7-26-8}.
To apply Theorem~\ref{th1-4-9}, we need to study the existence
and  non-degeneracy of the  critical
points of the K.-R. functions $\tilde{\mathcal K}_p$.

The domain geometry and topology will affect the existence of the critical point
for the following function in a domain $\Omega$:

\[
\sum_{j=1}^k\hat\kappa_j^2R_\Omega(x_j)
-\sum_{i\ne j}
\hat\kappa_i\hat\kappa_jG_\Omega(x_i,x_j),
\]
where $G_\Omega(y,x)$ is the Green's function in $\Omega$ satisfying the zero Dirichlet boundary
condition, and $R_\Omega(x)$ is its Robin function. See for example \cite{BPW,CaoPengYan2010,GrossiTakahashi2010}.
In the present situation, $\Omega=\mathbb R^1\times (0,\pi)$ is convex.
Hence, from \cite{GrossiTakahashi2010}, we know 
 that the above function has no critical point if all $\hat\kappa_j$ are positive.
Our result in this section shows that the periodic boundary condition  makes
the above function have critical point even if all  $\hat\kappa_j$ are positive.

 The first observation is that for any $x_j=(x_{j1},\frac\pi2)$, it holds

\[
\frac{\partial \tilde{\mathcal K}_p}{\partial x_{j2}}=0,\quad 
\frac{\partial^2 \tilde{\mathcal K}_p}{\partial x_{j2}\partial x_{m1}}=0.
\]
Thus, we just need to deal with the horizontal direction.

We let $\tilde \kappa_j= (-1)^{\tau_j}\hat\kappa_j$. For simplicity, we let $k=3$.
In the following, we fix $x_2=(0,\frac\pi2)$ and $\tilde \kappa_2=1$.  We denote
$x_1=(-a,\frac\pi2)$ and $x_3=(b,\frac\pi2)$.

Let $\Phi_L(r):=G_p((r,\frac\pi2),(0,\frac\pi2))$.  Then,

\begin{equation}\label{10-10-9}
\Phi_L'(r)=-\frac{F_L(r)}{2\pi},
\end{equation}
where

\begin{equation}\label{nn10-10-9}
F_L(r)=\sum_{n\in\mathbb Z}\frac1{\sinh(r+nL)}.
\end{equation}
We define
\begin{equation}\label{011-24-10}
g_L(a,b)
=
\tilde \kappa_1\Phi_L(a)
+\tilde \kappa_3\Phi_L(b)
+\tilde \kappa_1\tilde \kappa_3\Phi_L(a+b).
\end{equation}
Suppose that $(a_0,b_0)$ is a critical point of $g_L(a,b)$, then it holds
\[
F_L(a_0)+\tilde \kappa_3F_L(a_0+b_0)=0,\quad
F_L(b_0)+\tilde \kappa_1F_L(a_0+b_0)=0,
\]
from which, we obtain

\begin{equation}\label{1-10-9}
\tilde \kappa_1=-\frac{ F_L(b_0) }{F_L(a_0+b_0) },\quad \tilde \kappa_3=-\frac{ F_L(a_0) }{F_L(a_0+b_0) }.
\end{equation}
Hence we see that given any $(a_0,b_0)$, if $\tilde \kappa_1$ and $\tilde \kappa_3$
are defined as in \eqref{1-10-9}, then $(a_0,b_0)$ is a critical point of $g_L$.

 It is not difficult to check that $F_L(a)>0$ for $a\in (0,\frac L2)$ and $F_L(a)<0$ for $a\in (\frac L2, L)$.  Thus, we conclude that for $a_0,\, b_0\in (0,  \frac L2)$ and $a_0+b_0\in (\frac L2, L)$,
 then $\tilde \kappa_1>0$ and $\tilde \kappa_3>0$.  This shows that $\tau_j=0$, $j=1,2, 3$,
 and hence $\tilde{\mathcal K}_p$ has a critical point with  $\kappa_2=1$ and $\kappa_j$
 defined in \eqref{1-10-9}, which is positive.  
 
 On the other hand, if $a_0,\, b_0,\, a+b\in (0,  \frac L2)$,
 then $\tilde \kappa_1<0$ and $\tilde \kappa_3<0$. This gives that $\tau_1=1$ and $\tau_3=1$.
 
 \bigskip
 
We now discuss the solitary wave.  That is, we have  $L=+\infty$.
In this case, we take no background flow.  So $\gamma=0$.
Hence, we study the existence of critical points for the following function

 \begin{equation}\label{6-10-9}
\mathcal K(\mathbf{x})
=
\sum_{j=1}^k\hat\kappa_j^2R(x_j)
-\sum_{i\ne j}(-1)^{\tau_i+\tau_j}
\hat\kappa_i\hat\kappa_jG(x_i,x_j),
\end{equation}
where
\begin{equation}\label{060-24-10}
G(y,x)
=
\frac1{4\pi}
\ln\left(
1+\frac{2\sin x_2\sin y_2}
{\cosh(y_1-x_1)-\cos(y_2-x_2)}
\right),
\quad
R(x)=-\frac1{2\pi}\ln(2\sin x_2).
\end{equation}
The function $\mathcal K(\mathbf{x})$ governs the position of the vortex 
of a solitary wave.  The method developed in this paper can also be applied to construct asymmetric solitary waves in the capillary-gravity case, although we omit the proof.

We define
\begin{equation}\label{01-24-10}
g(a,b)
=
2\hat\kappa_1\Phi(a)+2\hat\kappa_3\Phi(b)
-2\hat\kappa_1\hat\kappa_3\Phi(a+b),
\end{equation}
where

\begin{equation}\label{n01-24-10}
\Phi(r)
=
\frac1{4\pi}
\log\left(1+\frac2{\cosh r-1}\right).
\end{equation}

Given $(a_0, b_0)$,  we 
solve the following equations
\begin{equation}\label{02-24-10}
-\frac{\hat\kappa_1}{\pi\sinh a}
+\frac{\hat\kappa_1\hat\kappa_3}{\pi\sinh(a+b)}=0,\quad
-\frac{\hat\kappa_3}{\pi\sinh b}
+\frac{\hat\kappa_1\hat\kappa_3}{\pi\sinh(a+b)}=0,
\end{equation}
 to obtain
\begin{equation}\label{03-14-10}
\hat\kappa_1=\frac{\sinh(a_0+b_0)}{\sinh b_0},
\qquad
\hat\kappa_3=\frac{\sinh(a_0+b)}{\sinh a_0}.
\end{equation}
With $\hat\kappa_2=1$ and  $\hat\kappa_1$ and $\hat\kappa_3$ defined in \eqref{03-14-10}, $(a_0, b_0)$
is a critical point of $g(a,b)$.

The above discussion shows that with the flexibility in the choice of the constant
$\hat\kappa_j$, it is quite easy to construct  $\mathcal K_p(\mathbf{x})$,
or $\mathcal K(\mathbf{x})$, such that
it has critical point,
which has no reflection symmetry.
The main objective in this section is to  check the non-degeneracy of the critical points
obtained in the above discussion.

\subsection{The periodic K.-R. function}

We take $\tau_1=\tau_2=\tau_3=0$,   $\kappa_2=1$, and
\begin{equation}\label{010-24-10}
a_0\in (0,\frac L4),\quad b_0=\frac L2-\frac{a_0}{3}.
\end{equation}
Then $a_0+b_0=\frac L2+\frac{2a_0}{3}>\frac L2$.
We assume that $a_0>0$ is small.

Recall that 
\begin{equation}\label{020-24-10}
\hat\kappa_1=-\frac{F_L(b_0)}{F_L(a_0+b_0)}>0,
\qquad
\hat\kappa_2=1,
\qquad
\hat\kappa_3=-\frac{F_L(a_0)}{F_L(a_0+b_0)}>0.
\end{equation}

We let
\[
A_L(r)
=\sum_{n\in\mathbb Z}
\frac{\cosh(r+nL)}{\sinh^2(r+nL)},
\quad
B_L(r):=\sum_{n\in\mathbb Z}\frac1{\sinh^2(r+nL)}.
\]
Set
\[
A_*=A_L\bigl(\frac L2\bigr)>0,\quad
B_*=B_L\bigl(\frac L2\bigr),\quad D_*=F_L'''\bigl(\frac L2\bigr),
\]

\[
\rho_L
=\partial_{x_2}^2R_p\bigl(0,\frac\pi2\bigr)
=
\frac1{2\pi}
+\frac1\pi\sum_{n\ne0}\frac1{\cosh(nL)+1}>0,
\]
and

\[
f_L
=-\frac16
-2\sum_{n=1}^{\infty}
\frac{\cosh(nL)}{\sinh^2(nL)},
\quad
\vartheta_L=f_L+\frac{2D_*}{27A_*}.
\]

 Direct computations show that 
\[
F_L(a)=a^{-1}+f_La+O(a^3),\quad \text{as}\; a\to 0+,
\]

and 
\[
F_L(\frac L2+\sigma)
=-A_*\sigma+\frac{D_*}{6}t^3+O(\sigma^5),\quad \text{as}\; \sigma\to 0.
\]
Thus, from \eqref{020-24-10}, we find that  as $a_0 \to0^+$,
\[
\begin{aligned}
\hat\kappa_1
&=\frac12+\frac{D_*}{36A_*}a_0^2+O(a_0^4),\\
\hat\kappa_3
&=\frac{3}{2A_*a_0^2}
+\frac{3\vartheta_L}{2A_*}+O(a_0^2).
\end{aligned}
\]

We use $H(a_0,\gamma)$ to denote the Hessian of the function $ \tilde{\mathcal K}_p$
at 
$x_1=(-a_0,\frac\pi2)$, $x_2=(0,\frac\pi2)$  and $x_3=(\frac L2-\frac{a_0}{3},\frac\pi2)$.
Then
\begingroup
\scriptsize
\setlength{\arraycolsep}{2pt}
\renewcommand{\arraystretch}{1.1}
%\begin{equation}%\label{eq:B2-Hessian-asymptotic}
\[
\begin{gathered}
H(a_0,\gamma)=
%\\[1.2ex]
\begin{pmatrix}
-\dfrac5{4\pi a_0^2}+O(1)
&
-\dfrac3{4\pi a_0^2}+O(1)
&0&0&0
\\[1.8ex]
-\dfrac3{4\pi a_0^2}+O(1)
&
-\dfrac9{4\pi a_0^2}+O(1)
&0&0&0
\\[1.8ex]
0&0
&
\dfrac5{4\pi a_0^2}+\gamma+O(1)
&
-\dfrac1{2\pi a_0^2}+O(1)
&
-\dfrac{3B_*}{4\pi A_*a_0^2}+O(1)
\\[1.8ex]
0&0
&
-\dfrac1{2\pi a_0^2}+O(1)
&
\dfrac2{\pi a_0^2}+2\gamma+O(1)
&
-\dfrac{3B_*}{2\pi A_*a_0^2}+O(1)
\\[2.8ex]
0&0
&
-\dfrac{3B_*}{4\pi A_*a_0^2}+O(1)
&
-\dfrac{3B_*}{2\pi A_*a_0^2}+O(1)
&
\begin{gathered}
\dfrac{9\rho_L}{4A_*^2a_0^4}
+\dfrac{9\rho_L\vartheta_L}{2A_*^2a_0^2}
\\[1ex]
+\dfrac9{4\pi a_0^2}
+\dfrac{3\gamma}{A_*a_0^2}
+O(1)
\end{gathered}
\end{pmatrix}
\end{gathered}
\]%\end{equation}
\endgroup

Let $H_{\rm h}(a_0)$ and $H_{\rm v}(a_0,\gamma)$ denote the two blocks in $H(a_0,\gamma)$.
For $a_0>0$ small, we have
\[
\det H_{\rm h}(a_0)
=
\frac{9}{4\pi^2a_0^4}+O(a_0^{-2})>0,
\]
and
\[
\det H_{\rm v}(a_0,0)
=
\frac{81\rho_L}{16\pi^2A_*^2a_0^8}
+O(a_0^{-6})>0,
\]
which gives
$
\det H(a_0,0)
>0$ if $a_0>0$ is small.
Thus, except for at most three $\gamma$, $\det H(a_0,\gamma)\ne 0$.

\subsection{The limiting K.-R. function}

 In this subsection, we discuss the critical points for the function
 $g(a, b)$ defined in \eqref{01-24-10}.
 In view of the non-existence result in \cite{GrossiTakahashi2010},
    we  choose $\tau_1=\tau_3=1$, and $\tau_2=0$  in \eqref{6-10-9}.
We take
$
a_0=tb_0$, $a_0>0$ is small, $t>0$ and $t\ne1.$
Here $t\ne1 $ gives an asymmetric configuration.
Given such $a_0$ and $b_0$, $\hat \kappa_1$ and $\hat \kappa_3$ are given in \eqref{03-14-10}.
Thus,
 $\mathcal K$ has a  critical point
$x_1=(-a_0,\frac\pi2)$, $x_2=(0,\frac\pi2)$  and $x_3=(b_0,\frac\pi2)$.
It is also a critical point of the following function

\[
 \tilde{\mathcal K}_{01}(\mathbf{x})=\mathcal K(\mathbf{x})-2\sum_{j=1}^3
 (-1)^{\tau_j} \hat\kappa_j\psi_{01}(x_j).
\]

In the following, we will prove that if
$a_0>0$ and $b_0>0$ are  small, this critical point for $\tilde{\mathcal K}_{01}(\mathbf{x})$ is  non-degenerate.
Consequently, in view of

\[
\tilde{\mathcal K}_p(\mathbf{x}) \to \tilde{\mathcal K}_{01}(\mathbf{x}),
\]
locally in $C^2$ as $L\to +\infty$,
  $\tilde{\mathcal K}_p$
also has "non-degenerate" critical point
$x_1=(-a_L,\frac\pi2)$, $x_2=(0,\frac\pi2)$  and $x_3=(b_L,\frac\pi2)$,
$a_L>0$ and $b_L>0$  small, if $L>0$ is large.

The Hessian of $\nabla^2\widetilde{\mathcal K}_{01}$ at  
$x_1=(-a_0,\frac\pi2)$, $x_2=(0,\frac\pi2)$  and $x_3=(b_0,\frac\pi2)$
is given by

\begingroup
\small
\setlength{\arraycolsep}{2.6pt}
\[
\nabla^2\widetilde{\mathcal K}
=
\begin{pmatrix}
\dfrac1{\pi\sinh^2a_0}
&
-\dfrac{\cosh(a_0+b_0)}{\pi\sinh a_0\sinh b_0}
&0&0&0
\\[2.2ex]
-\dfrac{\cosh(a_0+b_0)}{\pi\sinh a_0\sinh b_0}
&
\dfrac1{\pi\sinh^2b_0}
&0&0&0
\\[2.2ex]
0&0&A_{11}&A_{12}&A_{13}\\
0&0&A_{12}&A_{22}&A_{23}\\
0&0&A_{13}&A_{23}&A_{33}
\end{pmatrix},
\]
\endgroup
where
\[
\begin{aligned}
A_{11}
&=
\frac1{2\pi}
\left(
\hat\kappa_1^2-\frac2{\sinh^2a_0}
\right)-2\gamma\hat\kappa_1,
&
A_{12}
&=\frac{\hat\kappa_1}{\pi\sinh^2a_0},
\\
A_{22}
&=
\frac1{2\pi}
\left\{
1-2(\coth a_0+\coth b_0)^2
\right\}+2\gamma,
&
A_{23}
&=\frac{\hat\kappa_3}{\pi\sinh^2b_0},
\\
A_{33}
&=
\frac1{2\pi}
\left(
\hat\kappa_3^2-\frac2{\sinh^2b_0}
\right)-2\gamma\hat\kappa_3,
&
A_{13}
&=-\frac1{\pi\sinh a_0\sinh b_0}.
\end{aligned}
\]
Let $H_{\rm h}$ and $H_{\rm v}$ denote the horizontal and
vertical blocks.

For $b_0>0$ small, it holds
\[
\begin{aligned}
\hat\kappa_1
&=(t+1)
\left(
1+\frac{t^2+2t}{6}b_0^2+O(b_0^4)
\right),\\
\hat\kappa_3
&=\frac{t+1}{t}
\left(
1+\frac{2t+1}{6}b_0^2+O(b_0^4)
\right).
\end{aligned}
\]
 For fixed $t>0$ and $\gamma$, we expand
the determinants and obtain
\[
\begin{aligned}
\det H_{\rm h}
&=
-\frac{\sinh^2(a_0+b_0)}
{\pi^2\sinh^2a_0\,\sinh^2b_0}
=
-\frac{(t+1)^2}{\pi^2t^2}b_0^{-2}+O(1),
\\
\det H_{\rm v}
&=
-\frac{(t+1)^2}{2\pi^3t^4}b_0^{-2}
\bigl(t^2+t+1-6\pi t\gamma\bigr)
\bigl(t^2+t+1-4\pi t\gamma\bigr)
+O(1),
\\
\det\nabla^2\widetilde{\mathcal K}_{01}
&=
\frac{(t+1)^4}{2\pi^5t^6}b_0^{-4}
\bigl(t^2+t+1-6\pi t\gamma\bigr)
\bigl(t^2+t+1-4\pi t\gamma\bigr)
+O(b_0^{-2}).
\end{aligned}
\]
Thus, for each fixed $\gamma$, we can always find $t>0$, with $t\ne1$, such that 
$\det\nabla^2\widetilde{\mathcal K}_{01}\ne 0$.

\subsection{General case}
\label{subsec:B3-fixed-gaps}

We now fix
\[
x_1=\left(-a_0,\frac\pi2\right),\qquad
x_2=\left(0,\frac\pi2\right),\qquad
x_3=\left(b_0,\frac\pi2\right),
\]
where
\[
0<a_0,b_0<\frac L2,
\qquad 0<a_0+b_0<L,\qquad \gamma\in\mathbb R.
\]
In the previous  two subsections, we prove that if $a_0>0$ small and $ b_0=\frac L2-\frac{a_0}{3}$,
or if $a_0>0$ and $b_0>0$ small, the Hessian
$H(a_0, b_0,0)$ for  $ \tilde{\mathcal K}_p$ with $\gamma=0$ is non-degenerate.
 We have the following general result.

\begin{theorem}
 If $a_0+b_0\ne L/2$, then $H(a_0,b_0,0)$ is non-degenerate. Moreover,
except at most three $\gamma$,
$H(a_0,b_0,\gamma)$ is non-degenerate.
\end{theorem}

The proof of the above theorem involves the computations of some
elliptic integrals and is bit lengthy. 
Since the  two cases given in this section  are sufficient for our construction
of asymmetric water waves,  we omit its proof.  For results 
on the critical points of the K.-R. type functions involving the double
periodic boundary conditions,  the readers can refer to
\cite{ChaiLinWang2015,ChenLin2021,LinWang2010} and the references therein.
{
\section{The leading term of the  free surface}

In this section, we present the leading term of free-surface.  We focus  on the
pure-gravity case. 

The equation of the free boundary is given by $y_2=\pi +\Gamma_2(y_1,\pi)$, where $\Gamma_2$
satisfies \eqref{21-24-8}.  The leading term of  $\Gamma_2$
is determined by the solution of \eqref{81-24-8}, which can
be solved by the separation of variables technique as in Appendix~A.

 From \eqref{eq3-7}   and \eqref{3-8-5}, the leading
 term of $  \mathcal{U}_{\varepsilon,\mathbf{x},\mathbf{a}}$ is

\[
\sum\limits_{j=1}^k
(-1)^{\tau_j}\frac{2\pi \hat\kappa_j}{|\ln s_{\ep,j}|} G_p(x_j,y),
\]
where $
\hat\kappa_j=\kappa_j-(-1)^{\tau_j}\frac{\gamma\pi^2}{8}.
$
We also have
\[
\partial_{y_2}G_p
\left((s,\pi),(\xi,\frac\pi2)\right)
=-\frac1{2\pi}K_L(s-\xi),
\]
where $K_L(s)=\sum_{m\in\mathbb Z}\frac{1}{\cosh(s+mL)}$.

Let

\[
k_n=\frac{2n\pi}{L}.
\]

{\bf Case 1}. Set $\xi_1=a_0$, $\xi_2=0$ and $\xi_3= \frac L2- \frac{a_0}{3}$. Then,
$\tau_j=0$, $j=1, 2, 3$. 
The leading term  of the free surface is
\[
\pi-\frac{\pi^2\gamma}{L|\ln\varepsilon|}
\sum_{n=1}^{\infty}
\frac{
(\kappa_1-\frac{\gamma\pi^2}{8})\cos(k_n(y_1+a_0))+(1-\frac{\gamma\pi^2}{8})\cos(k_n y_1)+(\kappa_3-\frac{\gamma\pi^2}{8})\cos(k_n(y_1-\frac{L}{2}+\frac{a_0}{3}))}{
\left[g+\frac{\pi\gamma^2}{2}-\frac{\pi^2\gamma^2}{4}k_n\coth(\pi k_n)\right]\cosh(\pi k_n/2)
}.
\]

{\bf Case 2}.
Set $a_0=tb_0$, $\xi_1=a_L$, $\xi_2=0$ and $\xi_3= b_L$ for $L\gg1$, where $a_L\rightarrow a_0$ and $b_L\rightarrow b_0$ as $L\rightarrow \infty$. Then $\tau_1=\tau_3=1$, $\tau_2=0$.
The leading term of the free surface is
\[
\pi+\frac{\pi^2\gamma}{L|\ln\varepsilon|}
\sum_{n=1}^{\infty}
\frac{
\displaystyle (\kappa_1+\frac{\gamma\pi^2}{8})\cos(k_n(y_1+a_L))-(1-\frac{\gamma\pi^2}{8})\cos(k_n y_1)+(\kappa_3+\frac{\gamma\pi^2}{8})\cos(k_n(y_1-b_L))}{
\left[g+\frac{\pi\gamma^2}{2}-\frac{\pi^2\gamma^2}{4}k_n\coth(\pi k_n)\right]\cosh(\pi k_n/2)
}.
\]

\end{document}